\documentclass[11pt]{article}

\usepackage[tbtags]{amsmath}
\usepackage{amsthm,amsopn,amsfonts,amssymb,epsfig,anysize,tabmac}
\usepackage[affil-it]{authblk}
\usepackage{tikz}

\usepackage{epsfig,amsthm,amsopn,amsfonts,amssymb,graphicx,color}
\usepackage[backref]{hyperref}
\usepackage{verbatim}

\newtheorem{theorem}{Theorem}[section]
\newtheorem{lemma}[theorem]{Lemma}
\newtheorem{prop}[theorem]{Proposition}

\newtheorem{corollary}[theorem]{Corollary}

\numberwithin{equation}{section}

\theoremstyle{definition}
\newtheorem{example}[theorem]{Example}
\newtheorem{definition}[theorem]{Definition}

\theoremstyle{remark}
\newtheorem{remark}[theorem]{Remark}

\newlength{\cellsize}
\newcommand{\content}{\mathrm{con}}

\newcommand{\tS}{\tilde S_n^0}

\newcommand{\Gr}{\mathrm{Gr}}
\newcommand{\Fl}{\mathrm{Fl}}
\newcommand{\LC}{\mathcal{LC}}

\newcommand{\id}{\mathrm{id}}

\newcommand{\sign}{\mathrm{sign}}

\newcommand{\SSYT}{\mathrm{SSYT}}
\newcommand{\QH}{\mathrm{QH}}

\newcommand{\wt}{\mathrm{wt}}

\newcommand{\et}{\tilde{e}}
\newcommand{\ft}{\tilde{f}}
\newcommand{\st}{\tilde{s}}
\newcommand{\ve}{\varepsilon}
\newcommand{\vp}{\varphi}
\newcommand{\linv}{\mathrm{linv}}

\def \d {{\mathbf d}}

\def \la {\lambda}

\def\la {\lambda}

\def\shape{ {\rm {shape}}}

\definecolor{blue}{rgb}{.255,.41,.884} 
\definecolor{red}{rgb}{1, 0, 0} 
\definecolor{green}{rgb}{.196,.804,.196} 
\definecolor{yellow}{rgb}{1,.648,0} 
\definecolor{pink}{rgb}{1,0.5,0.5}

\edef\savecatcodeat{\the\catcode`@}
\catcode`\@=11

\def\tb@ifSpecChars#1#2{#1}
\def\tb@ifNoSpecChars#1#2{#2}

\def\tableau{%
  \bgroup
  \@ifstar{\let\Tif\tb@ifNoSpecChars\tb@tableauB}
          {\let\Tif\tb@ifSpecChars\tb@tableauB}}

\def\tb@tableauB{
  \@ifnextchar[{\tb@tableauC}{\tb@tableauC[]}}

\def\tb@tableauC[#1]{\hbox\bgroup%
    \let\\=\cr
    \def\bl{\global\let\tbcellF\tb@cellNF}%
    \def\tf{\global\let\tbcellF\tb@cellH}
    \dimen2=\ht\strutbox \advance\dimen2 by\dp\strutbox%
    \ifx\baselinestretch\undefined\relax%
    \else%
       \dimen0=100sp \dimen0=\baselinestretch\dimen0%
       \dimen2=100\dimen2 \divide\dimen2 by\dimen0%
    \fi%
    \let\tpos\tb@vcenter
    \tb@initYoung
    \tb@options#1\eoo
    \let\arrow\tb@arrow%
    \dimen0=\Tscale\dimen2%
    \dimen1=\dimen0 \advance\dimen1 by \tb@fframe%
    \lineskip=0pt\baselineskip=0pt
    \def\tb@nothing{}%
    \def\endcellno{$\rss\egroup\bss\egroup}
    \def\endcell{\endcellno\kern-\dimen0}
    \def\begincell{\vbox to\dimen0\bgroup\vss\hbox to\dimen0\bgroup\hss$}%
    \let\overlay\tb@overlay%
    \let\fl\tb@fl%
    \let\lss\hss\let\rss\hss\let\tss\vss\let\bss\vss
    \def\mkcell##1{
        \let\tbcellF\tb@cellD
        \def\tb@cellarg{##1}
        \ifx\tb@cellarg\tb@nothing\let\tb@cellarg\tb@cellE\fi%
        \begincell\tb@cellarg\endcellno
        \tbcellF}
    \let\savecellF\tbcellF
     \Tif{\catcode`,=4\catcode`|=\active}{}\tb@tableauD}%

\let\tb@savetableauD\tableauD
{
    \catcode`|=\active \catcode`*=\active \catcode`~=\active%
    \catcode`@=\active
\gdef\tableauD#1{%
  \Tif{
    \mathcode`|="8000 \mathcode`*="8000%
    \mathcode`~="8000 \mathcode`@="8000%
    \def@{\bullet}%
    \let|\cr
    \let*\tf
    \let~\sk
  }{}%
  \tpos{\tabskip=0pt\halign{&\mkcell{##}\cr#1\crcr}}%
  \global\let\tbcellF\savecellF
  \egroup
  \egroup}
}
\let\tb@tableauD\tableauD
\let\tableauD\tb@savetableauD
\let\tb@savetableauD\undefined

\def\tb@options#1{\ifx#1\eoo\relax\else\tb@option#1\expandafter\tb@options\fi}

\def\tb@option#1{%
  \if#1t\let\tpos\tb@vtop\fi
  \if#1c\let\tpos\tb@vcenter\fi
  \if#1b\let\tpos\vbox\fi
  \if#1F\tb@initFerrers\fi
  \if#1Y\tb@initYoung\fi
  \if#1s\tb@initSmall\fi
  \if#1m\tb@initMedium\fi
  \if#1l\tb@initLarge\fi
  \if#1p\tb@initPartition\fi
  \if#1a\tb@initArrow\fi
}

\def\tb@vcenter#1{\ifmmode\vcenter{#1}\else$\vcenter{#1}$\fi}

\def\tb@vtop#1{\hbox{\raise\ht\strutbox\hbox{\lower\dimen0\vtop{#1}}}}

\def\tb@initPartition{\def\Tscale{.3}}
\def\tb@initSmall{\def\Tscale{1}}
\def\tb@initMedium{\def\Tscale{2}}
\def\tb@initLarge{\def\Tscale{3}}

\def\tb@initArrow{\dimen2=1.25em}

\def\tb@initYoung{%
  \def\tb@cellE{}
  \let\tb@cellD\tb@cellN
  \def\sk{\global\let\tbcellF\tb@cellNF}}
\def\tb@initFerrers{%
  \def\tb@cellE{\bullet}
  \let\tb@cellD\tb@cellNF
  \def\sk{\bullet}}

\tb@initMedium

\def\tb@sframe#1{%
  \vbox to0pt{
    \vss
    \hbox to0pt{%
      \hss
      \vbox to\dimen1{
        \hrule depth #1 height0pt
        \vss
        \hbox to\dimen1{
          \vrule width #1 height\dimen1
          \hss
          \vrule width #1
          }%
        \vss
        \hrule height #1 depth 0in
        }%
      \kern-\tb@hframe
      }%
    \kern-\tb@hframe}}

\def\tb@hframe{.2pt}\def\tb@fframe{.4pt}\def\tb@bframe{2pt}
\def\tb@cellH{\tb@sframe{\tb@bframe}}       
\def\tb@cellNF{}                            
\def\tb@cellN{\tb@sframe{\tb@fframe}}       
\let\tbcellF\tb@cellN                       

\def\tb@rpad{1pt}
\def\tb@lpad{1pt}
\def\tb@tpad{1.8pt}
\def\tb@bpad{1.8pt}

\def\tb@overlay{\endcell\@ifnextchar[{\tb@overlaya}{\begincell}}
\def\tb@overlaya[#1]{\vbox to\dimen0\bgroup%
  \tb@overlayoptions#1\eoo%
  \tss\hbox to\dimen0\bgroup\lss$}
\def\tb@overlayoptions#1{\ifx#1\eoo\relax\else\tb@overlayoption#1\expandafter\tb@overlayoptions\fi}

\def\tb@overlayoption#1{
  \if#1t\def\tss{\vskip\tb@tpad}\let\bss\vss\fi
  \if#1c\let\tss\vss\let\bss\vss\fi
  \if#1b\def\bss{\vskip\tb@bpad}\let\tss\vss\fi
  \if#1l\def\lss{\hskip\tb@lpad}\let\rss\hss\fi
  \if#1m\let\lss\hss\let\rss\hss\fi
  \if#1r\def\rss{\hskip\tb@rpad}\let\lss\hss\fi
}

\def\tb@fl{\endcell\begincell\vrule depth 0pt width \dimen0 height \dimen0 \endcell\begincell}

\@ifundefined{diagram}{}{
\def\tb@arrowpad{.5}

\newoptcommand{\tb@arrow}{\@ne}[2]{%
  \endcell
   \begingroup%
   \let\dg@getnodesize\tb@getnodesize
   \dg@USERSIZE=#1\relax%
   \ifnum\dg@USERSIZE<\@ne \dg@USERSIZE=\@ne \fi%
   \dg@parse{#2}%
   \dg@label{\tb@draw{#1}{#2}}}

\def\tb@getnodesize#1#2#3#4#5{\dimen3=\tb@arrowpad\dimen2 #4=\dimen3 #5=\dimen3\relax}
\def\tb@getnodesize#1#2#3#4#5{\ifnum#2=0\ifnum#3=0\tb@getnodesizetail{#4}{#5}\else\tb@getnodesizehead{#4}{#5}\fi\else\tb@getnodesizehead{#4}{#5}\fi}
\def\tb@getnodesizetail#1#2{\dimen3=.5\dimen2 #1=\dimen3 #2=\dimen3}
\def\tb@getnodesizehead#1#2{\dimen3=.5\dimen2 #1=\dimen3 #2=\dimen3}

\def\tb@draw#1#2#3#4{%
        \dg@X=0\dg@Y=0\dg@XGRID=1\dg@YGRID=1\unitlength=.001\dimen0%
        \dg@LBLOFF=\dgLABELOFFSET \divide\dg@LBLOFF\unitlength%
        \dg@drawcalc
        \begincell
        \let\lams@arrow\tb@lams@arrow
        \begin{picture}(0,0)\begingroup\dg@draw{#1}{#2}{#3}{#4}\end{picture}%
        \endcell
        \endgroup
        \begincell}
}

\def\tb@lams@arrow#1#2{%
 \lams@firstx\z@\lams@firsty\z@
 \lams@lastx#1\relax\lams@lasty#2\relax
 \lams@center\z@
 \N@false\E@false\H@false\V@false
 \ifdim\lams@lastx>\z@\E@true\fi
 \ifdim\lams@lastx=\z@\V@true\fi
 \ifdim\lams@lasty>\z@\N@true\fi
 \ifdim\lams@lasty=\z@\H@true\fi
 \NESW@false
 \ifN@\ifE@\NESW@true\fi\else\ifE@\else\NESW@true\fi\fi
 \ifH@\else\ifV@\else
  \lams@slope
  \ifnum\lams@tani>\lams@tanii
   \lams@ht\ten@\p@\lams@wd\ten@\p@
   \multiply\lams@wd\lams@tanii\divide\lams@wd\lams@tani
  \else
   \lams@wd\ten@\p@\lams@ht\ten@\p@
   \divide\lams@ht\lams@tanii\multiply\lams@ht\lams@tani
  \fi
 \fi\fi
 \ifH@  \lams@harrow
 \else\ifV@ \lams@varrow
 \else \lams@darrow
 \fi\fi
}

\catcode`\@=\savecatcodeat
\let\savecatcodeat\undefined

\begin{document}

\title{Crystal approach to affine Schubert calculus}

\author{Jennifer 
Morse\footnote{supported by NSF grants DMS--1001898, DMS--1301695 and a Simons Fellowship.}
\; \; and \;
Anne Schilling\footnote{supported by NSF grants DMS--1001256, OCI--1147247 and a Simons Fellowship.}
}

%
%
%


\maketitle

\begin{abstract}
We apply crystal theory to affine Schubert calculus, Gromov-Witten
invariants for the complete flag manifold, and the positroid stratification 
of the positive Grassmannian.  We introduce 
operators on decompositions of elements in the type-$A$ affine Weyl group and 
produce a crystal reflecting the internal structure of the generalized Young modules 
whose Frobenius image is represented by stable Schubert polynomials.
We apply the crystal framework to products of a Schur function with a $k$-Schur 
function, consequently proving that a subclass of 3-point Gromov--Witten 
invariants of complete flag varieties for $\mathbb C^n$ enumerate 
the highest weight elements under these operators.
Included in this class are the Schubert structure constants in the
(quantum) product of a Schubert polynomial with a Schur function
$s_\lambda$ for all $|\lambda^\vee|< n$.  Another by-product gives
a highest weight formulation for various fusion coefficients of the 
Verlinde algebra and for the Schubert decomposition of certain
positroid classes. 
\end{abstract}

\section{Introduction}

\subsection{Background}
The theory of crystal bases was introduced by Kashiwara \cite{Kash:1990,Kash:1991}
in an investigation of quantized enveloping algebras $U_q(\mathfrak g)$
associated to a symmetrizable Kac--Moody Lie algebra $\mathfrak g$.
Integrable modules for quantum groups
play a central role in two-dimensional solvable lattice models.
When the absolute temperature is zero ($q=0$), there is a distinguished
{\it crystal basis} with many striking features.  
The most remarkable is that the internal structure of an integrable 
representation can be combinatorially realized by associating the 
basis to a colored oriented graph whose arrows are imposed by the 
Kashiwara (modified root) operators.
From the crystal graph, characters can be computed by enumerating elements
with a given weight and the tensor product decomposition
into irreducible submodules is encoded by
the disjoint union of connected components.  Hence, progress in
the field comes from having a natural combinatorial realization
of crystal graphs.

Schubert calculus is a theory whose development also hinges on combinatorial 
methods, but its origin is in geometry.  Initially motivated 
to determine the number of linear spaces of given dimension
satisfying certain geometric conditions,
the theory has grown to one that can address highly non-trivial curve
counting including the calculation of Gromov--Witten invariants.
The approach converts problems into computations with
representatives for Schubert classes in the (quantum)
cohomology ring of a flag variety.  Hence, the basic problem is 
one of producing and working with explicit representatives.

Crystal theory and Schubert calculus convene naturally 
in a foundational example.
From the geometric perspective, the problem is to count certain linear 
subspaces in projective space.  Developments in algebraic geometry and topology 
convert the problem into the computation of intersection numbers
of certain subvarieties in the Grassmannian $\Gr(a,n)$, 
which in turn are encoded 
by the structure constants of Schubert classes 
$\{\sigma_\lambda\}_{\lambda\subset \mathrm{(a^{n-a})}}$
for the cohomology ring $H^*(\Gr(a,n))$.
The computation is made concrete with a homomorphism $\psi$ 
from the ring $\Lambda$ of symmetric functions onto $H^*(\Gr(a,n))$.
In particular,
$$
\psi(s_\lambda) =\begin{cases} \sigma_\lambda & \text{if $\lambda\subset
(a^{n-a})$}\,,\\
0 & \text{otherwise}\,,
\end{cases}
$$
where $s_\lambda$ is a Schur function.
Intersection numbers thus sit as coefficients in 
the Schur expansion of a product of two Schur functions.

The related representation theoretic example is
$\mathfrak g = \mathfrak {gl}_n$.  The
heart of crystal theory realizes tensor multiplicities 
as highest weights (connected components) of a graph.
In this case, the crystal is the graph whose vertices are Young tableaux and 
whose edges are imposed by coplactic operators introduced by
Lascoux and Sch\"utzenberger~\cite{LS:1978,LS:1979}.
The number of connected components is the
multiplicity $c_{\lambda,\mu}^\nu$ of the irreducible highest weight
module $V_\nu$ in $V_\lambda\otimes V_\mu$:
$$
	V_\lambda\otimes V_\mu = \bigoplus_\nu c_{\lambda,\mu}^\nu V_\nu\,.
$$
A look back to the early 1900s discovery that the Frobenius image 
of $V_\lambda$ is the Schur function $s_\lambda$ shows that the 
Grassmannian intersection numbers are $c_{\lambda,\mu}^\nu$ as well.
An explicit rule to compute $c_{\lambda,\mu}^\nu$ by counting
a subclass of Young tableaux was formulated by Littlewood and Richardson 
in 1934~\cite{LR:1934}, but the first proof only 
arrived 40 years later with Sch\"utzenberger~\cite{Schuetz:1977}.

This example provides a template within which the representation theory
of different modules and the geometry of other varieties may be investigated.  
Even incremental variations inspire highly intricate combinatorics 
and leave unanswered questions.  For example, Schubert calculus 
of the flag manifold $\Fl_n$ 
is highly developed --- from a construction of Schubert classes $\{\sigma_w\}$ 
indexed by elements in the symmetric group $S_n$
by Bernstein-Gelfand-Gelfand~\cite{BGG:1973} and Demazure~\cite{Dem:1974}, 
to the explicit identification of the classes with polynomial 
representatives introduced by Lascoux and Sch\"utzenberger~\cite{LS:1982}.  
Nevertheless, an LR rule for the constants in 
\begin{equation}
\label{schubcon}
	\sigma_u \cup \sigma_w = \sum_{v\in S_n} c_{u,w}^v\,\sigma_{v}
\end{equation}
has yet to be discovered.
Related efforts are summarized in Section~\ref{sec:related}
and confirm that it is not for lack of trying.

The main thrust of this article is to introduce crystals into a
generalization of Schubert calculus centered around the {affine Grassmannian} 
$\Gr = G(\mathbb C((t)))/G(\mathbb C[[t]])$ for $G=SL(n,\mathbb C)$,
where $\mathbb C[[t]]$ is the ring of formal power series
and $\mathbb C((t)) = C[[t]][t^{-1}]$ is the ring of formal Laurent series.  
Quillen (unpublished) and Garland and Raghunathan~\cite{GR:1975} showed 
that $\Gr$ is homotopy-equivalent to the group $\Omega SU(n,\mathbb C)$
of based loops into $SU(n, \mathbb C)$.  Consequently, its homology and 
cohomology acquire algebra structures.  In particular,
it follows from Bott~\cite{Bott:1958} that $H_*(\Gr)$ 
and $H^*(\Gr)$ can be identified
with a subring $\Lambda_{(n)}$ and a  quotient $\Lambda^{(n)}$
of the ring $\Lambda$ of symmetric functions.
On one hand, using the algebraic nil-Hecke ring construction,
Kostant and Kumar~\cite{KK:1986} studied Schubert bases of $H^*(\Gr)$ 
and Peterson~\cite{Peterson:1997} studied Schubert bases of $H_*(\Gr)$,
$$
\{\xi^w\in H^*(\Gr,\mathbb Z) \mid w\in \tilde S_n^0\}\;\; \,\text{and}\,\;\;
\{\xi_w\in H_*(\Gr,\mathbb Z) \mid w\in \tilde S_n^0\} 
\,.
$$
These are indexed by the subset of the affine symmetric group 
$\tilde{S}_n = \langle s_0, s_1, \ldots, s_{n-1} \rangle$ 
consisting of affine Grassmannian elements ---
representatives of minimal length in cosets of $\tilde{S}_n/S_n$.

On the other hand, a distinguished basis for $\Lambda_{(n)}$
comprised of elements called {\it $k$-Schur functions},
$\{ s_w^{(k)} \mid w\in \tilde S_n^0\}$,
came out of a study \cite{LLM:2003} of Macdonald polynomials
($k=n-1$).  
It was shown in \cite{LM:2008} that
the (affine-LR) coefficients in the products 
\begin{equation}
\label{klr}
s_u^{(k)}\,s_w^{(k)} = \sum_{v\in \tS} c_{u,w}^{v,k}
\, s_v^{(k)}
\end{equation}
contain all structure constants (Gromov-Witten invariants)
for a quantum deformation of the cohomology 
of the Grassmannian (e.g. \cite{Kont,Witt}).  
A basis (of {\it dual $k$-Schur functions})
for $\Lambda^{(n)}$ was also introduced in \cite{LM:2008} and
therein generalized to a family $\{\mathcal F_{vw^{-1}}\}$ 
that alternately encodes the constants by
\begin{equation}
\label{dualks}
\mathcal F_{vw^{-1}}
=\sum_{u\in\tS} c_{u,w}^{v,k}\,\mathcal F_{u}
\,.
\end{equation}
The two approaches converged when Lam proved~\cite{Lam:2008}
that the $k$-Schur basis is a set of representatives for the Schubert 
classes of $H_*(\Gr)$ and the Schubert structure constants in 
homology exactly match the affine-LR coefficients (see also~\cite{LLMSSZ:2014}
for more details).

\subsection{Crystals and the affine Grassmannian}

In this article, we produce a crystal in the affine framework
that has applications to affine-LR coefficients and to several 
other families of elusive constants. 
The $k$-Schur functions can be characterized using decreasing factorizations
of elements in the type-$A$ affine Weyl group.  We introduce a set of 
operators (see Section~\ref{subsection.definition}) that act on a subclass of
these factorizations and prove that the resulting graph is a 
$U_q(A_{\ell-1})$-crystal using the Stembridge local axioms~\cite{St:2003} 
(see Theorem~\ref{theorem.crystal}). 
At a basic level, we find that the crystals support generalized Young--Specht 
modules of $S_n$ associated to permutation diagrams.  
These are the modules whose Frobenius images are stable Schubert polynomials 
$F_w$ (also known as Stanley symmetric functions).  In Theorem~\ref{theorem.intertwine}, 
we show that the Edelman--Greene  decomposition into irreducible characters
intertwines with our crystal operators.

We prove that the enumeration of highest weight elements 
in the crystal are $k$-Schur coefficients in the product of a Schur 
and a $k$-Schur function.
These are in fact affine LR-coefficients \eqref{klr} as it is
known that there is an element $w\in\tS$ where
$s_\mu=s^{(k)}_w$ anytime $\mu\subset (r^{n-r})$, for $1\leq r<n$.
A translation operator (defined in \eqref{Ri}) enables us to 
generalize the framework within which we can apply the crystal.
We include two proofs detailing this application;
one using crystal theory and another that extends the
Remmel--Shimozono involution on tableaux~\cite{RS:1998}.  

\medskip

\noindent
{\bf (Theorem~\ref{theorem.main})}.  
{\it Let $v,w\in \tS$ and $\mu\subset (r^{n-r})$ for some $1\leq r<n$.
If $\ell(v)-\ell(w) \neq |\mu|$, then $c_{Rw_\mu,w}^{Rv,k}=0$. Otherwise,
if $vw^{-1}\in S_{\hat x}$ or $\ell(\mu)=2$,
\[
	c_{Rw_\mu,w}^{Rv,k}=\#\text{ of highest weight factorizations of $vw^{-1}$ of weight $\mu$},
\]
where $S_{\hat x}$ denotes a finite subgroup of $\tilde S_n$ generated by
a strict subset of $\{s_0,\ldots,s_{n-1}\}$ and $R$ is a product of $k$-rectangle translation operators.
}

The crystal also connects to several families of intensely studied constants that
arise as a subset of affine LR-coefficients.  We discussed the
genus 0, 3-point Gromov--Witten invariants of Grassmannians.  In fact,
our results apply more generally to the complete flag manifold.
Quantum cohomology was defined for any K\"ahler algebraic manifold $X$.  
When $X=\Fl_n$, as a linear space, $\QH^*(\Fl_n) = H^*(\Fl_n)\otimes 
\mathbb Z[q_1,\ldots,q_{n-1}]$ for parameters $q_1,\ldots,q_{n-1}$. 
However, the multiplicative structure is defined by (where $w_0$ is the longest element in $S_n$)
\begin{equation}
	\sigma_u *_q \sigma_w = 
	\sum_v \sum_{\d} \langle u,w, v\rangle_{\d} \,q^{\d}\,
	\sigma_{w_0 v}\;,
\end{equation}
where the structure constants are the 3-point Gromov--Witten invariants
of genus 0, constants which count equivalence classes of certain rational curves 
in $\Fl_n$. Peterson asserted that $\QH^*(G/P)$ of a flag variety is
a quotient of the homology $H_*(\Gr_G)$ of the affine Grassmannian 
up to localization (details carried out in \cite{LS:2010}).
Consequently, $\langle u, w, v\rangle_{\d}$ arise as coefficients in \eqref{klr}
and in particular, when $\d=0$, these include the Schubert structure constants 
of \eqref{schubcon}.

\medskip

\noindent
{\bf (Theorem~\ref{the:GW})}
{\it
For any $\d\in\mathbb N^{n-1}$ and $u,w,v\in S_n$ 
where $u$ is Grassmannian with descent at position $r$, 
if $(Rv)w^{-1}\in S_{\hat x}$ for some $x\in[n]$, then
$$
\langle u,w,v\rangle_{\d}=\# \;\text{of
highest weight factorizations 
of $(Rv)w^{-1}$ of weight $\mu$,}
$$
where $R$ is a translation defined in Theorem~\ref{the:GW}.
}

\medskip

Subclasses of the affine LR-coefficients also include all Schubert
structure constants \eqref{schubcon} and all constants for the
Verlinde (fusion) algebra of the Wess--Zumino--Witten 
model associated to $\widehat{su}(\ell)$ at level $n-\ell$.  
That is, it was shown in \cite{LM:2008} that affine-LR coefficients 
contain 
the fusion coefficients $\mathcal N_{\lambda, \mu}^{\nu}$,  defined for
$\lambda,\mu,\nu\subseteq (n-\ell)^{\ell-1}$ by
\begin{equation}
\label{equation.fusion}
L(\lambda)\otimes_{n-\ell} L(\mu) = 
\bigoplus_\nu
\mathcal N_{\lambda, \mu}^{\nu}
L(\nu)
\,,
\end{equation}
where the fusion product $\otimes_{n-\ell}$ is the reduction of the tensor
product of integrable representations with highest weight $\lambda$ and $\mu$
via the representation at level $n-\ell$ of $\widehat{su}(\ell)$.

A family of affine Stanley symmetric functions, 
indexed by affine elements $\widetilde w\in\tilde S_n$, was
introduced in \cite{Lam:2006}.  It was shown that these functions 
$\{F_{\tilde w}\}_{\tilde w\in\tilde S_n}$ reduce to dual $k$-Schur functions 
indexed by affine Grassmannian elements, 
to the stable Schubert polynomials when $\tilde w$ is a finite element of
$S_n$, and to cylindric Schur functions of Postnikov \cite{Posttoric:2005}
when $\tilde w$ has no braid relation.
We discovered that there is a correspondence between the set $\tilde S_n$ of 
affine elements and certain skew shapes $\nu/\lambda$ (see Definition~\ref{skewdualgrass})
and in fact prove that any affine Stanley $F_{\widetilde w}$ is 
a skew dual $k$-Schur function $\mathcal F_{w_\nu w_\lambda^{-1}}$.

This in hand, we connect our results to a finer subdivision 
of the Grassmannian than the usual Bruhat decomposition
called the {\it positroid stratification}~\cite{Post:2005}.
Its complexification was proven \cite{KLS:2013} to coincide 
with the projection of the Richardson decomposition of the 
flag manifold \cite{Lus:1998,Rie:2006,BGY:2006}
and it was shown that each cohomology class is the image under $\psi$ of
an affine Stanley symmetric function.  Because cylindric Schur functions 
give access to Gromov--Witten invariants for the Grassmannian and
they are contained in the set $\{F_{\tilde w}\}$, it is
implied in \cite{KLS:2013} that a subset of the positroid varieties
relates to quantum cohomology of Grassmannians.  We extend their result, 
proving that Gromov--Witten invariants for the complete flag 
manifold arise in the Schubert decomposition of the cohomology 
class of any positroid variety.  We then apply the crystal on 
affine factorizations to this study.

\subsection{Outline}

Basic notation is reviewed in Section~\ref{section.notation}
regarding the affine Weyl group, crystals, and the Stembridge local axioms~\cite{St:2003}.  
The crystal operators on affine factorizations for certain $\widetilde w\in \tilde S_n$
are introduced in Section~\ref{section.crystal operators}.
We discuss their properties and the theorem that the resulting 
graph $B(\widetilde w)$ is a crystal in the category of integrable highest weight crystals 
for type $A$.  A proof that the Stembridge axioms hold is 
relegated to Appendix~\ref{appendix.stembridge}.  Section~\ref{section.specht} shows that $B(\widetilde w)$ supports
the generalized Young--Specht module of a permutation diagram associated to $\widetilde w$.
In Section~\ref{section.gromov witten}, we 
connect affine Stanley symmetric functions to dual $k$-Schur functions and
use the crystal to describe various $k$-Schur 
structure constants in the product of a Schur function 
and a $k$-Schur function.  The subsections relate these coefficients
to families of constants that arise from quantum flag varieties, WZW fusion, 
Schur times Schubert polynomials, and positroid varieties.  
In Section~\ref{section.involution}, we produce a sign-reversing involution on $B(\widetilde{w})$
that refines the Remmel and Shimozono~\cite{RS:1998} proof of the 
classical LR-rule.  A second proof of Theorem~\ref{theorem.main} 
arises consequently.

\subsection{Related work}
\label{sec:related}
Manifestly positive combinatorial formulas for structure coefficients for the Verlinde fusion algebra,
or the quantum cohomology of the Grassmannian, and for the full flag have been actively sought for
some time now.  In the fusion case, Tudose~\cite{Tudose:2002} gave a 
combinatorial interpretation when $\mu$ has at most two columns in her thesis. 
For $n=2,3$, positive formulas are known~\cite{BMW:1992,BSVW:2014} as well as when 
$\lambda$ and $\mu$ are rectangles~\cite{SS:2001}. Korff and Stroppel~\cite{KorffStroppel:2010} 
give a formula using the plactic algebra, however their formula involves signs. Subsequently,
Korff~\cite{Korff:2011,Korff:2013} gave a new algorithm for the calculation of the fusion coefficients using
relations to integrable models. In~\cite{LM:2008}, it was shown that the fusion and three point 
Gromov--Witten invariants for Grassmannians form a special case of the $k$-Schur function 
structure coefficients. The fusion case of Theorem~\ref{theorem.main} was treated in~\cite{MS:2012}.

Knutson formulated a conjecture for the quantum Grassmannian Littlewood--Richardson coefficients in terms of 
puzzles~\cite{KnutsonTao:2003} as presented in~\cite{BKT:2003}.  Coskun~\cite{Coskun:2009} gave 
a positive geometric rule to compute the structure constants of the cohomology ring of two-step flag varieties
in terms of Mondrian tableaux.  A proof of the puzzle conjecture was recently given by Buch et al.~\cite{BKPT:2014}.
In the flag case, Fomin, Gelfand and Postnikov~\cite{FGP:1997} computed the quantum Monk rule which
was extended in~\cite{Post:1999} to the quantum Pieri rule. Berg, Saliola and Serrano~\cite{BSS:2013}
computed the Littlewood--Richardson coefficients for $k$-Schur functions for the case which is equivalent
to the quantum Monk rule. Denton~\cite{Denton:2012} proved a special $k$-Littlewood--Richardson rule
when there is a single term without multiplicity.

The Schur times (quantum) Schubert polynomial coefficients fall within the realm of Theorem~\ref{the:GW}
and have received much attention during the last years. Lenart~\cite{Lenart:2004,Lenart:2010} used growth diagrams
and plactic relations to approach this problem. Benedetti and Bergeron~\cite{BB:2012, BB:2013} relate the Schur times Schubert
problem to $k$-Schur function structure coefficients using strong order. The Schur times quantum Schubert coefficients
are addressed by M\'esz\'aros, Panova and Postnikov~\cite{MPP:2012} using the Fomin--Kirillov algebra in the hook
and two-row case. As we will see in Section~\ref{section.gromov witten} this is the opposite extreme from the 
cases treated in Theorem~\ref{the:GW}.

\subsection*{Acknowledgments}
We would like to thank the ICERM program ``Automorphic Forms, Combinatorial Representation Theory and 
Multiple Dirichlet Series" during the spring 2013 and IHES in Orsay for hospitality, where part of this work was done.
Both authors would like to thank the Simons Foundation for sabbatical support.

Many thanks to Avi Dalal, Nate Gallup and Mike Zabrocki for their help 
with the implementation of weak tableaux and $k$-charge during Sage Days 
49 in Paris. This work benefitted from  computations
with {\sc Sage}~\cite{sage,sage-combinat}. Finally, we would like to thank 
Sara Billey, Dan Bump, Patrick Clarke, Adriano Garsia, Thomas Lam, Luc Lapointe, 
Mark Shimozono, and Josh Swanson for enlightening discussions.

\section{Preliminaries}
\label{section.notation}

Here we review background on affine permutations and crystals 
that will be used throughout the article.  Otherwise, definitions
and notation will be introduced as needed.  In particular, 
our convention for partitions and tableaux is summarized
in Section~\ref{section.specht}.

\subsection{Extended affine symmetric group}
\label{subsection.affine permutations}

Fix $n\in \mathbb{Z}_{> 0}$.  
The {\it affine symmetric group} $\tilde S_n$ is the Coxeter
group generated by $\langle s_0, s_1,\ldots, s_{n-1}\rangle$ satisfying 
the relations
\begin{equation}
\label{coxeter}
\begin{aligned}
s_i^2 &= 1 && \text{for all $i$,} \\
s_is_{i+1}s_i &= s_{i+1}s_is_{i+1} && \text{for all $i$,}\\
s_is_j &= s_js_i && \text{if}\; |i-j|> 1\,,
\end{aligned}
\end{equation}
where indices are taken modulo $n$ (we will work mod $n$ without further comment).
The subgroup generated by $s_1,\ldots,s_{n-1}$ is isomorphic to the symmetric 
group $S_n$.  A word $i_1 i_2\cdots i_m$ in the alphabet $[n]=\{0,1,\ldots,n-1\}$ 
corresponds to the affine permutation $w=s_{i_1} \cdots s_{i_m} \in \tilde S_{n}$. 
The length $\ell(w)$ of $w\in\tilde S_n$ is defined by the length of 
its shortest word.  Any word of this length is said to be {\it reduced}.

There is a concrete realization of $\tilde S_n$ as the
affine Weyl group $\tilde A_{n-1}$~\cite{Lus:1983}.  Affine permutations
are bijections $w$ from $\mathbb Z\to\mathbb Z$ where $w(i+n)=w(i)+n$ 
for all $i$ and where 
\begin{equation}
\label{sum0}
\sum_{i=1}^n(w(i)-i)=0\,.
\end{equation}
Since an affine permutation $w$ is determined by its tuple of values
$[ w(1), w(2), \ldots, w(n)]$, we often use only this
{\it window} to represent it.
The length of $w$ can be determined by counting the appropriate notion of
inversions.  In particular, the {\it left inversion vector}
$\linv(w)=(\alpha_1,\ldots,\alpha_n)$ 
is the composition where $\alpha_i$ records the number of positions 
$-\infty<j<i$ such that $w(j)>w(i)$.  
It was proven~\cite{Shi:1986,BB:1996} that
\begin{equation}
\label{equation.length}
\ell(w)=|\linv(w)|\,.
\end{equation}

We shall also have the need to work in a larger setting with
{\it extended} affine permutations, bijections as before
but without requiring condition \eqref{sum0}.  The set of 
these elements forms the {\it extended affine symmetric group}
which can be realized by adding a generator
$\tau$ to $\tilde S_n$ where $\tau(i)=(i+1)$.  It is subject to 
the relation $\tau s_i = s_{i+1}\tau$.  For any extended affine 
permutation $w$, there is a unique non-negative integer $r$ where 
$w=\tau^r v$ and $v\in\tilde S_n$.  Note then that
\begin{equation}
\label{remark.shift}
v= [w(1)-r, w(2)-r,\ldots,w(n)-r]\,.
\end{equation}
The extended affine symmetric group contains $\tilde S_n$ 
as a normal subgroup.  Its coset decomposition is 
given by subsets $\tilde S_{n,r}$ made up of 
elements $w$ with the property that $\sum_{i=1}^n (w(i)-i)=rn$.  

The set $\tilde S_n^0$ of \textit{affine Grassmannian elements} is
the set of minimal length coset representatives of $\tilde S_n/S_n$.
Representatives are given by those $w\in \tilde S_n$ 
for which all reduced words end in $0$.
An element $w$ is affine Grassmannian if and only if its window
is increasing.  We shall also define an extended
affine permutation $w$ to be affine Grassmannian when 
$w(1) < w(2) < \cdots < w(n)$.

\subsection{Kashiwara crystals and Stembridge local axioms}
\label{subsection.stembridge}

Kashiwara~\cite{Kash:1990,Kash:1994} introduced a {\it crystal}
as an edge-colored directed graph satisfying a simple set of axioms.
Let $\mathfrak{g}$ be a symmetrizable Kac--Moody algebra with associated root, 
coroot and weight lattices $Q,Q^\vee,P$. Let $I$ be the index set of the Dynkin 
diagram and denote the simple roots, simple coroots and fundamental weights 
by $\alpha_i$, $\alpha^\vee_i$ and $\Lambda_i$ ($i\in I$), respectively.
There is a natural pairing $\langle \cdot,\cdot \rangle \colon Q^\vee\otimes P \rightarrow
\mathbb{Z}$ defined by $\langle \alpha_i^\vee, \Lambda_j\rangle=\delta_{ij}$.

\begin{definition}
\label{definition.crystal}
An {\it abstract $U_q(\mathfrak{g})$-crystal} is a nonempty set $B$ together with maps
\begin{equation*}
\begin{split}
	\wt &\colon B \to P\\
	\et_i, \ft_i &\colon B \to B \cup \{\bf{0}\} \qquad \text{for all $i\in I$}
\end{split}
\end{equation*}
satisfying
\begin{enumerate}
\item $\ft_i(b)=b'$ is equivalent to $\et_i(b')=b$ for $b,b'\in B$, $i\in I$.
\item For $i\in I$ and $b\in B$
\begin{equation*}
\begin{split}
	&\wt(\et_i b) = \wt(b) +\alpha_i \qquad \text{if $\et_ib \in B$},\\
	&\wt(\ft_i b) = \wt(b) -\alpha_i \qquad \text{if $\ft_ib \in B$}.
\end{split}
\end{equation*}
\item For all $i\in I$ and $b\in B$, we have
$
	\vp_i(b) = \ve_i(b) + \langle \alpha_i^\vee, \wt(b)\rangle,
$
where
\begin{equation}
\label{equation.ve vp}
\begin{split}
	\ve_i(b) &= \max\{d \ge 0 \mid \et_i^d(b) \neq \bf{0}\},\\
	\vp_i(b) &= \max\{d \ge 0 \mid \ft_i^d(b) \neq \bf{0}\}.
\end{split}
\end{equation}
\end{enumerate}
\end{definition}

\begin{remark}
Although the above axioms are sometimes used to define only semi-normal crystals,
this suffices here since we consider crystals coming from $U_q(\mathfrak{g})$-representations, 
all of which are semi-normal. 
\end{remark}

\begin{remark}
The axioms of Definition~\ref{definition.crystal} define an edge-colored directed graph with vertex set $B$
by drawing an edge $b \stackrel{i}{\rightarrow} b'$ when $\ft_i(b) = b'$.
\end{remark}

Abstract crystals do not necessarily correspond to crystals coming from 
$U_q(\mathfrak{g})$-representations. Stembridge~\cite{St:2003} provided a simple set of local
axioms that uniquely characterize the crystals corresponding 
to representations of simply-laced algebras. We briefly review
his axioms here.

Let $A=[a_{ij}]_{i,j\in I}$ be the Cartan matrix of a simply-laced 
Kac--Moody algebra $\mathfrak{g}$ (off-diagonal entries are either 0 or -1).
In this paper we mainly consider the Cartan matrix of type $A_{\ell-1}$. 
An edge-colored graph $X$ is called $A$-regular if
it satisfies the following conditions (P1)-(P6), (P5'), and (P6'):
\begin{enumerate}
\item[(P1)] 
All monochromatic directed paths in $X$ have finite length. In particular $X$ has 
no monochromatic circuits.
\item[(P2)] 
For every $i\in I$ and every vertex $x$, there is at most one edge 
$y \stackrel{i}{\longrightarrow} x$ and at most one edge 
$x\stackrel{i}{\longrightarrow} z$.
\end{enumerate}
We introduce the notation
\begin{equation*}
\Delta_i \ve_j(x)=\ve_j(x)-\ve_j(\et_i x), \qquad
\Delta_i \vp_j(x)=\vp_j(\et_i x)-\vp_j(x),
\end{equation*}
whenever $\et_i x$ is defined, and
\begin{equation*}
\nabla_i \ve_j(x)=\ve_j(\ft_i x)-\ve_j(x), \qquad
\nabla_i \vp_j(x)=\vp_j(x)-\vp_j(\ft_i x),
\end{equation*}
whenever $\ft_i x$ is defined, where $\ve_i$ and $\vp_i$ are defined
as in~\eqref{equation.ve vp}.

For fixed $x\in X$ and a distinct pair $i,j\in I$, 
assuming that $\et_i x$ is defined, require
\begin{enumerate}
\item[(P3)] $\Delta_i \ve_j(x) + \Delta_i \vp_j(x) = a_{ij}$, and
\item[(P4)] $\Delta_i \ve_j(x)\le 0$, $\Delta_i \vp_j(x)\le 0$.
\end{enumerate}
Note that for simply-laced algebras $a_{ij}\in \{0,-1\}$ for $i,j\in I$ distinct.
Hence (P3) and (P4) allow for only three possibilities:
\begin{equation*}
	(a_{ij},\Delta_i\ve_j(x),\Delta_i\vp_j(x))=(0,0,0),(-1,-1,0),(-1,0,-1).
\end{equation*} 

Assuming that $\et_ix$ and $\et_jx$ both exist, we require
\begin{enumerate}
\item[(P5)] $\Delta_i\ve_j(x)=0$ implies $y:=\et_i\et_jx=\et_j\et_ix$ and
$\nabla_j\vp_i(y)=0$.
\item[(P6)] $\Delta_i\ve_j(x)=\Delta_j\ve_i(x)=-1$ implies
$y:=\et_i \et_j^2 \et_i x= \et_j \et_i^2 \et_j x$ and $\nabla_i\vp_j(y)=\nabla_j\vp_i(y)=-1$.
\end{enumerate}

Dually, assuming that $\ft_ix$ and $\ft_jx$ both exist, we require
\begin{enumerate}
\item[(P5')] $\nabla_i\vp_j(x)=0$ implies $y:=\ft_i \ft_j x = \ft_j \ft_i x$ and
$\Delta_j\ve_i(y)=0$.
\item[(P6')] $\nabla_i\vp_j(x)=\nabla_j\vp_i(x)=-1$ implies
$y:=\ft_i \ft_j^2 \ft_i x= \ft_j \ft_i^2 \ft_j x$ and $\Delta_i\ve_j(y)=\Delta_j
\ve_i(y)=-1$.
\end{enumerate}

Stembridge proved~\cite[Proposition 1.4]{St:2003} that any two $A$-regular
posets $P,P'$ with maximal elements $x,x'$ are isomorphic if and only if
$\vp_i(x)=\vp_i(x')$ for all $i\in I$. 
Moreover, this isomorphism is unique.
Let $\lambda = \sum_{i\in I} \mu_i \Lambda_i$.
Denote by $B(\lambda)$ the unique $A$-regular poset with maximal element $b$ 
such that $\vp_i(b)=\mu_i$ for all $i\in I$.

\begin{theorem} \cite[Theorem 3.3]{St:2003} 
\label{thm:stembridge}
If $\mathfrak{g}$ is a simply-laced Kac--Moody Lie algebra with Cartan matrix $A$, 
then the crystal graph of the irreducible $U_q(\mathfrak{g})$-module 
of highest weight $\lambda$ is $B(\lambda)$.
\end{theorem} 

From now on we call crystals corresponding to $U_q(\mathfrak{g})$-modules simply $U_q(\mathfrak{g})$-crystals.
As Theorem~\ref{thm:stembridge} shows, for simply-laced types,
it can be checked whether a crystal is a $U_q(\mathfrak{g})$-crystal
by checking axioms (P1)-(P6').

An element $u\in B$ is called \textit{highest weight} if $\et_iu=\bf{0}$ for all $i\in I$.
A crystal $B$ is in the category of highest weight integrable crystals if for every $b\in B$, there exists a sequence
$i_1,\ldots,i_h\in I$ such that $\et_{i_1} \cdots \et_{i_h}b$ is highest weight.
One of the most important applications of crystal theory is that 
crystals are well-behaved with respect to taking tensor products.

\begin{theorem} \cite{Kash:1991,N:1993,Littelmann:1994}
\label{theorem.highestwts}
If $B$ is a $U_q(\mathfrak{g})$-crystal in the category of integrable highest-weight crystals,
then the connected components of $B$ correspond to the irreducible components and the
irreducible components are in bijection with the highest weight vectors.
\end{theorem}

\section{Crystal on affine factorizations}
\label{section.crystal operators}

We start this section by defining affine factorizations, 
elements of $\tilde S_n$ with a decreasing feature
that appear prominently in the geometry and combinatorics of
the affine Grassmannian $\Gr$.  We introduce operators on a 
distinguished subset of these elements and prove properties needed 
to show that they are crystal operators for quantum algebra representations 
of type $A$.
In Section~\ref{subsection.two factor}, we show how the crystal operators can
be extended to act on a different subset of affine factorizations.

In subsequent sections, we shall see that the operators support
certain Young--Specht modules and that they specialize to the reflection, 
raising and lowering crystal operators of~\cite{KN:1994} 
on semi-standard Young tableaux.
We will also give applications of the resulting crystal graph 
to the affine and positive
Grassmannian, Gromov--Witten invariants, and fusion rules.

\subsection{Affine factorizations}
\label{subsection.afffac}

Let $i_1\cdots i_\ell$ be a sequence with each $i_r\in [n]$.
The word $i_1\cdots i_\ell$ is {\it cyclically decreasing} if no number is 
repeated and $j-1\,j$ does not occur as a subword for any $j\in[n]$ (recall that we
take all indices mod $n$).
If $i_1\cdots i_\ell$ is cyclically decreasing, then we say
the permutation $w = s_{i_1}\cdots s_{i_\ell}$ is cyclically decreasing.
Define the {\it content} of a permutation $w \in \tilde S_n$ as
$$
\content(w) = \{i \in [n] \mid i \;\text{appears in a reduced word for} \;w\} \,.
$$
Note that this set can be obtained from a single reduced word
for $w$ and is independent of the reduced word chosen.
Moreover, a cyclically decreasing permutation
is uniquely determined by its content. Hence, we often abuse notation and
write the cyclically decreasing words for the actual permutation.

Foremost, cyclically decreasing elements describe the structure in homology
$H_*(\Gr)$.  For any $u\in\tS$, the Pieri rule is
\begin{equation}
\label{kpieri}
\xi_{s_{r-1}\cdots s_1 s_0}\xi_u = \sum_{v:\ell(v)=r} \xi_{vu}\,,
\end{equation}
over all cyclically decreasing permutations $v\in \tilde S_n$ where
$\ell(vu)=r+\ell(u)$ and $vu\in\tS$.
The focal point of our study is a set of distinguished products of 
cyclically decreasing elements.  For any composition 
$\alpha=(\alpha_1,\ldots,\alpha_{\ell})\in \mathbb N^\ell$
and $w\in\tilde S_n$ of length $|\alpha|:=\alpha_1+\cdots + \alpha_{\ell}$,
an {\it affine factorization} of $w$ of weight $\alpha$ is a decomposition
of the form $w=w^{\ell}\cdots w^1$, where $w^i$ 
is a cyclically decreasing permutation of length $\alpha_i$ for each
$1\le i\le \ell$.  We denote the set of
affine factorizations of $w$ by $\mathcal W_{w}$, and the subset of
these having weight $\alpha$ is $\mathcal W_{w,\alpha}$.
Their enumeration describes more general homology products;
for $u\in \tS$,
\begin{equation}
\label{iteratekpieri}
\xi_{s_{\alpha_\ell-1}\cdots s_1 s_0}\cdots
\xi_{s_{\alpha_1-1}\cdots s_1 s_0}\xi_u = 
\sum_{v\in\tS}\mathcal K_{vu^{-1},\alpha}\,
\xi_{v}\,,
\end{equation}
where $\mathcal K_{w,\alpha}=|\mathcal W_{w,\alpha}|$ for any $w\in \tilde S_n$.

The generating functions of affine factorizations were 
considered by Lam in~\cite{Lam:2006} as {\it affine Stanley 
symmetric functions}.  Defined for any $\tilde{w} \in\tilde S_n$ by
\begin{equation}
\label{affstan}
F_{\tilde{w}}(x)=F_{\tilde{w}}=\sum_{w^\ell\cdots w^1\in \mathcal W_{\tilde{w}}}
x_1^{\ell(w^1)}\cdots x_\ell^{\ell(w^\ell)} \,,
\end{equation}
the functions connect to several notable families.  At the fundamental level, 
when $w\in S_n$, these are precisely the functions constructed by Stanley in~\cite{Stanley:1984} 
with the specific intention of realizing the number of reduced words for $w$ as the
coefficient of $x_1x_2\cdots x_\ell$. These ``Stanley symmetric functions"
had in fact been studied earlier by Lascoux and Sch\"utzenberger \cite{LS:1982} 
as the stable limit of Schubert polynomials $\mathfrak S_w(x)$.  
More generally, we will prove that
affine Stanley symmetric functions are none other than the dual $k$-Schur functions
\eqref{dualks} of~\cite{LM:2008} and we will
discuss the tie between $F_{\tilde w}$ and cohomology classes of 
positroid varieties \cite{Post:2005,KLS:2013}.

\subsection{The crystal operators}
\label{subsection.definition}

We define operators $\et_r,\ft_r,$ and $\st_r$ that act on the $r$-th
and $(r+1)$-st factors in an affine factorization by altering the contents 
of these consecutive factors.  The alteration is determined by a process 
of pairing reflections in their respective contents.  The process is independent 
of $r$ and it can thus be defined on a product of two cyclically decreasing 
factors.

Given a cyclically decreasing permutation $u\in \tilde S_n$,
since $\content(u)$ is strictly contained in $[n]$,
there exists some $x\in [n]$ such that $u\in S_{\hat x}$, where we have defined
$$
	S_{\hat x}=\langle s_0,s_1,\ldots,\hat s_x,\ldots,s_{n-1}\rangle\subseteq \tilde S_n \,.
$$
Therefore, for such a fixed $x$, there is a unique reduced 
word for $u$ given by the decreasing arrangement of entries 
in $\content(u)$ taken with respect to the order
\begin{equation}
\label{orderx}
	x-1> x-2>\dots>0> n-1>\dots>x+1\,.
\end{equation}
Consider cyclically decreasing permutations $u,v\in S_{\hat x}$.  
The $uv$-{\it pairing with respect to $x$} is defined by pairing
the largest $b\in \content(u)$ with the smallest $a>b$ in $\content(v)$
using the ordering in~\eqref{orderx}. If there is no such $a$ in $\content(v)$ then
$b$ is unpaired.  The pairing proceeds in decreasing order
on elements of $\content(u)$, and with each iteration 
previously paired letters of $\content(v)$ are ignored.

\begin{example}
\label{ex:pairing}
Let $n=14, u=s_{12} s_5 s_9 s_8 s_2$, and $v=s_7 s_6 s_4 s_1 s_0 s_{13} s_{11}$.
The $uv$-pairing with respect to $x=10$ proceeds from left to right on the 
words for $u$ and $v$ given by writing $\content(u)$ and $\content(v)$ in
decreasing order with respect to $x$: 
$$
(9,8,5_1,2_2,12_3) (7,6_1,4_2,1,0,13_3,11)\,.$$
Here the pairs are denoted by matching subscripts.
The $uv$-pairing with respect to $x=3$ is
$$
(2,12_1,9_2,8_3,5_4) (1,0_3,13_1,11_2,7,6_4,4)
\,.
$$
For $n=5,u=s_1s_0$, and $v=s_4s_3s_1$, the $uv$-pairing with respect to 2 is 
$$
(1,0_1)(1_1,4,3)
\,.
$$
\end{example}

Given $w\in S_{\hat x}$ for some $x\in [n]$, the crystal operators are defined to act by 
changing unpaired entries in adjacent factors of a factorization $w^\ell\cdots w^1$ of $w$.
Since all the factors in an affine factorization of $w$ lie in $S_{\hat x}$, we can pair 
any two adjacent factors with respect to $x$ and set
$$
L_r(w^\ell \cdots w^1)=
\{
b\in\content(w^{r+1}) \mid
b  \text{ is unpaired in the $w^{r+1}w^r$-pairing}
\}
\,,
$$
$$
R_r(w^\ell \cdots w^1)=
\{
b\in\content(w^{r}) \mid
b  \text{ is unpaired in the $w^{r+1}w^r$-pairing}
\}
\,.
$$

\begin{definition}
\label{definition.crystal operators}
Fix $x\in[n]$.  We define operators on cyclically decreasing 
$u,v\in S_{\hat x}$ as follows:

\medskip

\noindent
(i) $\et_1(uv)=\tilde u \tilde v$ where $\tilde u$ and $\tilde v$  are the
unique cyclically decreasing elements with
$$
\content(\tilde u)=\content(u)\backslash\{b\}\quad\text{and} 
\quad \content(\tilde v)=\content(v)\cup\{b-t\}
$$
for $b=\min(L_1(uv))$ and
$t=\min\{i\geq 0\mid b-i-1\not\in\content(u)\}$.
If $L_1(uv)=\emptyset$, $\et_1(uv)=\bf{0}$.

\medskip

\noindent
(ii) $\ft_1(uv)=\tilde u\tilde v$ 
where $\tilde u$ and $\tilde v$  are the
unique cyclically decreasing elements with
$$
\content(\tilde u)=\content(u)\cup\{a+s\}\quad\text{and}\quad
\content(\tilde v)=\content(v)\backslash\{a\}
$$ 
for $a=\max(R_1(uv))$ and
$s=\min\{i\geq 0\mid a+i+1\not\in\content(v)\}$.
If $R_1(uv)=\emptyset$, $\ft_1(uv)=\bf{0}$.

\medskip

\noindent
(iii) $\st_1=\ft_1^{q-p}$ if $q>p$ and $\st_1=\et_1^{p-q}$ if $p>q$
where $p=|L_1(uv)|$  and $q=|R_1(uv)|$.
When $p=q$, $\st_1$ is the identity map.

\medskip

\noindent
Given an affine factorization $w^{\ell} \cdots w^1$ for $w\in S_{\hat x}$,
$\et_r,\ft_r,\st_r$ are defined for $1\leq r<\ell$ by
\begin{equation}
\label{equation.er}
\begin{split}
	\et_r(w^\ell\cdots w^1) &= w^{\ell}\cdots \et_1(w^{r+1}w^r)\cdots w^1 \,,\\
	\ft_r(w^\ell\cdots w^1) &= w^{\ell}\cdots \ft_1(w^{r+1}w^r)\cdots w^1 \,,\\
	\st_r(w^\ell\cdots w^1) &= w^{\ell}\cdots \st_1(w^{r+1}w^r)\cdots w^1\,.
\end{split}
\end{equation}
\end{definition}

\begin{remark}
Definition~\ref{definition.crystal operators}
is well-defined since the cyclically decreasing permutations
$\tilde u$ and $\tilde v$ are uniquely defined by a strict subset 
of $[n]$ giving their contents.  In particular,
$\content(v)\cup \{b-t\}\subseteq[n]\backslash\{x\}$
since $x\not\in\content(u)\cup\content(v)$ 
and $b-t\in\content(u)$ and similarly,
$\content(u)\cup \{a+s\}\subseteq[n]\backslash\{x\}$.
\end{remark}

\begin{example}
We appeal to the pairings computed in Example~\ref{ex:pairing} to
compute the images of the following affine factorizations.
With $n=14, u=s_{12} s_5 s_9 s_8 s_2$, and $v=s_7 s_6 s_4 s_1 s_0 s_{13} s_{11}$,
and $x=10$, we have
$$
\et_1(uv)= (9,5_1,2_2,12_3) (8,7,6_1,4_2,1,0,13_3,11)
\,,
$$
$$
\ft_1(uv)= (9,8,7,5_1,2_2,12_3) (6_1,4_2,1,0,13_3,11)
\,,
$$
$$
\st_1(uv)= (9,8,7,5_1,2_2,1,12_3) (6_1,4_2,0,13_3,11)
\,.
$$
For the same $u$ and $v$, but now with $x=3$, we have
$$
\et_1(uv)= (12_1,9_2,8_3,5_4) (2,1,0_3,13_1,11_2,7,6_4,4)\,.
$$
Pairing $uv=(s_1s_0)(s_1s_4s_3)\in \tilde S_5$ with respect to $x=2$ yields
$\et_1(uv)= (0_1)(1_1,0,4,3)$.
\end{example}

For $x\in[n]$ and any $w\in S_{\hat x}$, consider the graph $B(w)$ whose 
vertices are the affine factorizations $\mathcal{W}_w$ with $\ell$ factors (some of which might
be trivial) and whose $I$-colored edges $x\stackrel{i}{\to} y$ for $x,y\in B(w)$ are determined by 
$$
	\ft_i\, x=y \,.
$$
In Appendix~\ref{appendix.stembridge}, we show that $B(w)$ is a $U_q(A_{\ell-1})$-crystal graph by proving
that the Stembridge axioms spelled out in Section~\ref{subsection.stembridge} are satisfied. 

\begin{theorem}
\label{theorem.crystal}
For $x\in[n]$ and any $w\in S_{\hat x}$, $B(w)$ is a $U_q(A_{\ell-1})$-crystal.
\end{theorem}

Consequently, by Theorem~\ref{theorem.highestwts}, the connected components 
of $B(w)$ are in bijection with highest weight vectors as defined below.

\begin{definition}
\label{def:yama}
Fix $x\in [n]$, $w\in S_{\hat x}$, and a composition $\alpha=(\alpha_1,\ldots,\alpha_\ell)$ 
with $|\alpha|=\ell(w)$.  The factorization $w^\alpha\in \mathcal{W}_{w,\alpha}$
is {\it highest weight} when $\et_r w^\alpha={\bf 0}$ for all $1\leq r<\ell$.
That is, $w^\alpha$ is highest weight if there is no unpaired residue in $w^{r+1}$
in the $w^{r+1}w^{r}$-pairing with respect to $x$ for every 
$r=1,\ldots,\ell-1$.
\end{definition}

\begin{figure}
\begin{center}
\scalebox{0.8}{
\begin{tikzpicture}[>=latex,line join=bevel,]
\node (s3+s4*s1+s2) at (30bp,157bp) [draw,draw=none] {$\left(s_{3}, s_{4}s_{1}, s_{2}\right)$};
  \node (s3*s1+s4*s2+1) at (74bp,9bp) [draw,draw=none] {$\left(s_{3}s_{1}, s_{4}s_{2}, 1\right)$};
  \node (1+s3*s1+s4*s2) at (80bp,305bp) [draw,draw=none] {$\left(1, s_{3}s_{1}, s_{4}s_{2}\right)$};
  \node (s3*s1+s2+s4) at (190bp,157bp) [draw,draw=none] {$\left(s_{3}s_{1}, s_{2}, s_{4}\right)$};
  \node (s3*s1+1+s4*s2) at (110bp,157bp) [draw,draw=none] {$\left(s_{3}s_{1}, 1, s_{4}s_{2}\right)$};
  \node (s1+s3+s4*s2) at (190bp,305bp) [draw,draw=none] {$\left(s_{1}, s_{3}, s_{4}s_{2}\right)$};
  \node (s3+s1+s4*s2) at (80bp,231bp) [draw,draw=none] {$\left(s_{3}, s_{1}, s_{4}s_{2}\right)$};
  \node (s3*s1+s4+s2) at (74bp,83bp) [draw,draw=none] {$\left(s_{3}s_{1}, s_{4}, s_{2}\right)$};
  \node (s1+s3*s2+s4) at (190bp,231bp) [draw,draw=none] {$\left(s_{1}, s_{3}s_{2}, s_{4}\right)$};
  \draw [red,->] (1+s3*s1+s4*s2) ..controls (80bp,284.87bp) and (80bp,264.8bp)  .. (s3+s1+s4*s2);
  \definecolor{strokecol}{rgb}{0.0,0.0,0.0};
  \pgfsetstrokecolor{strokecol}
  \draw (89bp,268bp) node {$2$};
  \draw [blue,->] (s1+s3+s4*s2) ..controls (190bp,284.87bp) and (190bp,264.8bp)  .. (s1+s3*s2+s4);
  \draw (199bp,268bp) node {$1$};
  \draw [red,->] (s3+s4*s1+s2) ..controls (42.167bp,136.54bp) and (54.608bp,115.61bp)  .. (s3*s1+s4+s2);
  \draw (67bp,120bp) node {$2$};
  \draw [red,->] (s3+s1+s4*s2) ..controls (88.25bp,210.65bp) and (96.617bp,190.01bp)  .. (s3*s1+1+s4*s2);
  \draw (108bp,194bp) node {$2$};
  \draw [blue,->] (s3*s1+s4+s2) ..controls (74bp,62.872bp) and (74bp,42.801bp)  .. (s3*s1+s4*s2+1);
  \draw (83bp,46bp) node {$1$};
  \draw [blue,->] (s3+s1+s4*s2) ..controls (66.099bp,210.43bp) and (51.766bp,189.21bp)  .. (s3+s4*s1+s2);
  \draw (70bp,194bp) node {$1$};
  \draw [blue,->] (s3*s1+1+s4*s2) ..controls (100.1bp,136.65bp) and (90.059bp,116.01bp)  .. (s3*s1+s4+s2);
  \draw (106bp,120bp) node {$1$};
  \draw [red,->] (s1+s3*s2+s4) ..controls (190bp,210.87bp) and (190bp,190.8bp)  .. (s3*s1+s2+s4);
  \draw (199bp,194bp) node {$2$};
\end{tikzpicture}
}
\end{center}
\caption{The crystal $B(w)$ for $w=s_3s_4s_1s_2\in S_5$ with 3 factors (of type $A_2$)
\label{figure.crystal}}
\end{figure}
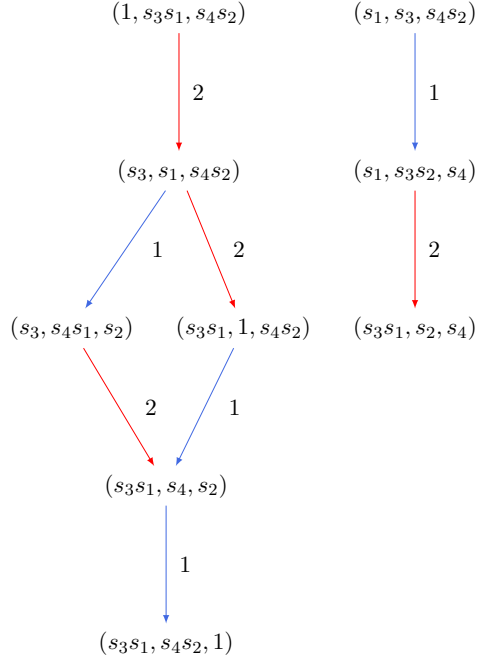

\begin{example}
\label{example.crystal}
The crystal $B(s_3s_4s_1s_2)$ of type $A_2$ is displayed in 
Figure~\ref{figure.crystal}.  It has two highest weight elements 
$(1)(s_3 s_1)(s_4 s_2)$ and $(s_1)(s_3)(s_4s_2)$ of weights
$(2,2)$ and $(2,1,1)$, respectively. In {\sc Sage},
this crystal can be generated by
\begin{verbatim}
  sage: W = WeylGroup(['A',4], prefix='s')
  sage: w = W.from_reduced_word([3,4,1,2])
  sage: B = crystals.AffineFactorization(w,3)
  sage: view(B)
\end{verbatim}
\end{example}

\subsection{Properties of the crystal operators}
\label{subsection.properties}
Here we establish that the crystal operators of 
Definition~\ref{definition.crystal operators}
map factorizations of $w$ to factorizations of the same element $w$
so that indeed 
$\et_i, \ft_i \colon \mathcal W_w \to \mathcal W_w \cup \{ {\bf 0} \}$.
To this end, we carefully study properties of our pairing process.

\begin{lemma}
\label{lem:uvdecomp}
Consider $x\in [n]$ and cyclically decreasing permutations $u,v\in S_{\hat x}$ 
where there exists $b=\min(L_1(uv))$.  For the cyclically decreasing elements
$u_1,u_2,v_1,$ and $v_2$ defined by

\noindent
$\content(u_1)=\{z\in\content(u) \mid b<z<x\},$ 
$\content(u_2)=\{z\in\content(u) \mid  x+1 \le z<b-t-1\}$

\noindent
$\content(v_1)=\{z\in\content(v) \mid b+1<z<x\}$,
$\content(v_2)=\{z\in\content(v) \mid x+1\le z<b-t\}$,

\noindent
where $t=\min\{i\geq 0 \mid b-i-1\not\in\content(u)\}$,
we have the decompositions
\begin{equation}
\label{uvdecomp}
u=u_1\,(s_b\, s_{b-1} \cdots s_{b-t}) \,u_2
\quad\text{and}\quad
v=v_1\,(s_b \cdots s_{b-t+1}) v_2\,.
\end{equation}
If $\ell(uv)=\ell(u)+\ell(v)$, then $uv$ is the product 
of the cyclically decreasing elements
\begin{equation}
\label{tuv}
\tilde u = u_1 (s_{b-1}\, \cdots s_{b-t})u_2
\quad\text{and}\quad
\tilde v = v_1 (s_b \cdots s_{b-t+1})s_{b-t}\,v_2
\end{equation}
and, under the $uv$-pairing, every element of $\content(v_1)$ 
is paired with something in $\content(u_1)$ and
every element of $\content(u_2)$ is paired with something in 
$\content(v_2)$.
\end{lemma}

\begin{proof}
Given that $b=\min(L_1(uv))$ exists, $b\in\content(u)$ and we have
$b\neq x$.  Therefore, 
$u=u_1\,(s_b\, s_{b-1} \cdots s_{b-t}\, \hat{s}_{b-t-1}) \,u_2$
as claimed.
Since the pairing process proceeds from largest to smallest
on entries of $\content(u)$, and $b\in\content(u)$ is unpaired,
$b+1\not\in\content(v)$ and 
every element of $\content(v_1)$ is paired with something in
$\content(u_1)$.
Further, since $b$ is the {\it smallest} unpaired element,
$b-1,\ldots,b-t$ are paired with $b,\ldots,b-t+1\in\content(v)$ 
and every element in $\content(u_2)$ is paired with something
in $\content(v_2)$.

Note since ${b+1}\not\in\content(v)$ by definition of $b$,
$v=v_1\,(s_b \cdots s_{b-t+1}) v_2$
for $\content(v_2)=\{z\in\content(v) \mid x+1\le z\leq b-t\}$.
Equipped also with
$u=u_1\,(s_b\, s_{b-1} \cdots s_{b-t}\, \hat{s}_{b-t-1}) \,u_2$,
we can commute to obtain
$$
u=u_1\,u_2\,(s_b\, s_{b-1} \cdots s_{b-t})
\quad\text{and}\quad
v=(s_b \cdots s_{b-t+1})v_1 v_2
\,.
$$
A succession of Coxeter relations~\eqref{coxeter} implies that
\[
	(s_b\, s_{b-1} \cdots s_{b-t})
	(s_b \cdots s_{b-t+1})=
	(s_{b-1}\, \cdots s_{b-t})
	(s_b \cdots s_{b-t}).
\]
Therefore, $uv=UV$ for
$U = u_1 u_2 (s_{b-1}\, \cdots s_{b-t})$
and $V=
(s_b \cdots s_{b-t+1})s_{b-t}\,v_1\,v_2\,. $
The conditions on the content of $u_2$ 
in Lemma~\ref{lem:uvdecomp} allow us again to commute 
to find that $U=u_1 (s_{b-1}\, \cdots s_{b-t})u_2=\tilde u$,
where we see that $\tilde u$ is cyclically decreasing,
$\content(\tilde u)=\content(u)\backslash\{b\}$,
and $\ell(\tilde u)=\ell(u)-1$.

On the other hand, the length of $V$ is at most $\ell(v)+1$
since it only differs from $v$ by the additional generator
$s_{b-t}$.  If we assume that $\ell(uv)=\ell(\tilde u V)=\ell(u)-1+\ell(v)+1$,
then equality $\ell(V)=\ell(v)+1$ must hold.
This given, $b-t\not\in\content(v_1)\cup \content(v_2)$.
Since $x+1\le b-t\le b$ and $\content(v_1)=\{z\in\content(v) \mid b+1<z<x\}$,
we can commute to find that
$V= v_1 (s_b \cdots s_{b-t+1})s_{b-t}\,v_2=\tilde v$
is cyclically decreasing. 
\end{proof}

The next lemma follows in the same fashion.
\begin{lemma}
\label{lem:Ruvdecomp}
Consider $x\in [n]$ and cyclically decreasing permutations $u,v\in S_{\hat x}$ 
where there exists $a=\max(R_1(uv))$.  For the cyclically decreasing 
elements $u_1,u_2,v_1,$ and $v_2$ defined by 

\noindent
$\content(u_1)=\{z\in\content(u) \mid a+s< z<x\},$ 
$\content(u_2)=\{z\in\content(u) \mid  x+1 \le z<a-1\}$

\noindent
$\content(v_1)=\{z\in\content(v) \mid a+s+1<z<x\}$,
$\content(v_2)=\{z\in\content(v) \mid x+1\le z< a\}$

\noindent
where $s=\min\{i\geq 0 \mid a+i+1\not\in\content(v)\}$,
we have the decompositions
\begin{equation}
\label{maxR}
u=u_1\,(s_{a+s-1}\, s_{a+s-2} \cdots s_{a}) \,u_2
\quad\text{and}\quad
v=v_1\,(s_{a+s} \cdots s_{a+1}s_a) v_2\,.
\end{equation}
If $\ell(uv)=\ell(u)+\ell(v)$, then $uv$ is the product 
of the cyclically decreasing elements
\begin{equation}
\tilde u = u_1 (s_{a+s} s_{a+s-1}\, \cdots s_{a})u_2
\quad\text{and}\quad
\tilde v = v_1 (s_{a+s} \cdots s_{a+1})\,v_2
\end{equation}
and, under the $uv$-pairing,
every element of $\content(v_1)$ is paired with something in 
$\content(u_1)$ and
every element of $\content(u_2)$ is paired with something in 
$\content(v_2)$.
\end{lemma}

The fundamental task of crystal operators is to send a factorization 
of $w$ to another factorization of $w$, with a carefully incremented weight change.
From now on we fix $\ell = \ell(\beta)$ to be the length of all weights, 
where if necessary some parts of $\beta$ might be zero.
Let $\alpha_r$ be the $r$-th simple root of type $A_{\ell-1}$.
This given, we can specify the weight change under the
crystal operators of Definition~\ref{definition.crystal operators} 
and show they are inverses of each other.

\begin{prop}
\label{prop:stayput}
Fix $x\in [n]$ and $w\in S_{\hat x}$. If
$w^\beta:=w^\ell \cdots w^1 \in \mathcal W_{w,\beta}$,
then for any $1\le r<\ell$,  
\begin{enumerate}
\item $\et_r(w^\beta) \in \mathcal W_{w,\beta+\alpha_r}$
and $\ft_r(w^\beta)\in \mathcal W_{w,\beta-\alpha_r}$,
or $w^\beta$ is annihilated,
\item $\st_r (w^\beta) \in \mathcal W_{w,s_r \beta}$
where $s_r$ acts on $\beta$ by interchanging $\beta_r$ and $\beta_{r+1}$,
\item $\varepsilon_r(w^\beta)=|L_r(w^\beta)|$ and
$\varphi_r(w^\beta) = |R_r(w^\beta)|$,
\item $\ft_r\,\et_r(w^\beta)=w^\beta$, or $w^\beta$ is annihilated.
The same is true for  $\et_r\,\ft_r(w^\beta)$.
\end{enumerate}
\end{prop}

\begin{proof}
Fix $x\in [n]$ and $w\in S_{\hat x}$.
By the definition of $\et_r,\ft_r,\st_r$ for any $r$ given in
\eqref{equation.er}, it suffices to consider
$uv\in \mathcal W_{w,(\beta_1,\beta_2)}$ 
where $\ell(v)=\beta_1$, $\ell(u)=\beta_2$ and $\ell(uv)=\beta_1+\beta_2$ and to prove 
$$
	\st_1 (uv) \in \mathcal W_{w,(\beta_2,\beta_1)}
\;,
\quad
	\et_1(uv) \in \mathcal W_{w,(\beta_1+1,\beta_2-1)} 
\;,
\quad
	\ft_1(uv)\in \mathcal W_{w,(\beta_1-1,\beta_2+1)}\,,
$$
or $uv$ is annihilated.

To this end, if $L_1(uv)=\emptyset$, $\et_1$ annihilates $uv$. 
Otherwise, $b=\min(L_1(uv))$ exists and by Lemma~\ref{lem:uvdecomp},
$w=uv=\tilde u\tilde v$ where $\tilde u$ and $\tilde v$ are
cyclically decreasing permutations with
$\content(\tilde u)=\content(u)\backslash\{b\}$
and $\content(\tilde v)=\content(v)\cup\{b-t\}$.
In fact, $\et_1(uv)=\tilde u\tilde v$ by the 
definition of $\et_1$. 
Note that $\ell(\tilde u)=\beta_2-1$ since it is
obtained by deleting one generator from the
cyclically decreasing permutation $u$.
On the other hand, $\tilde v$ is obtained by adding
one generator to $v$ and therefore
$\ell(\tilde v)\leq \beta_1+1$.
By assumption $\ell(uv)=\beta_1+\beta_2
= \ell(\tilde u\tilde v)\leq \ell(\tilde u)+\ell(\tilde v)
\leq \beta_2-1+\beta_1+1$.
Therefore, $\ell(\tilde v)=\beta_1+1$ and we have proven
$\et_1(uv)=\tilde u\tilde v\in
\mathcal W_{uv,(\beta_1+1,\beta_2-1)}$.

It is also clear from the above discussion that all unbracketed letters 
in $u_1$ in $uv$ remain unbracketed in $\tilde u \tilde v$ implying 
that $\varepsilon_1(uv) = | L_1(uv)|$.
Other cases in (2) and (3) follow in a similar manner.

To prove (4), again consider $uv\in\mathcal W_{w,(\beta_1,\beta_2)}$.
If $L_1(uv)=\emptyset$, then $\et_1$ annihilates $uv$. Otherwise let
$b=\min(L_1(uv))$. For $\tilde u\tilde v= \et_1(uv)$,
recall the decompositions of \eqref{tuv}:
\begin{equation}
\tilde u = u_1 (s_{b-1}\, \cdots s_{b-t})\,u_2\quad
\text{and}\quad
\tilde v=
v_1\,(s_b \cdots s_{b-t+1}s_{b-t})\,v_2\,.
\end{equation}
Proceed with the $\tilde u\tilde v$-pairing on 
the largest to smallest entries  of $\content(\tilde u)$.
Every entry in $\content(v_1)$ is paired to 
something in $\content(u_1)$ by Lemma~\ref{lem:uvdecomp}.
Next, $b-1,\ldots,b-t\in\content(\tilde u)$
are paired with $b,\ldots,b-t+1\in \content(\tilde v)$
and we find that $b-t\in\content(\tilde v)$ is unpaired.
Therefore, $\max(R_1(\tilde u\tilde v))=b-t$.
The conditions on $\content(v_1)$ from
Lemma~\ref{lem:uvdecomp} tell us that
$b+1\not\in\content(\tilde v)$ implying by the definition
of $\ft_1$ that $\ft_1\et_1(uv)=uv$.

The proof for $\et_1\ft_1$ follows in a similar manner.
\end{proof}

\begin{example}
Let $n=8$, $u=s_4s_3s_2s_1s_0s_7$ and $v=s_5s_2s_1s_0$. Since the
$uv$-pairing with respect to $x=6$ is
$uv=(4_1,3,2,1_2,0_3,7_4)(5_1,2_2,1_3,0_4)$, we find that
\begin{equation*}
	\et_1(uv) =  (4_1,3,1_2,0_3,7_4)(5_1,2_2,1_3,0_4,7)\;.
\end{equation*}
It is not hard to check that $\ft_1$ acts on this by
deleting the 7 from the right factor and adding a 2 to the
left with braid relations, and indeed $\ft_1 \et_1(uv) = uv$.
\end{example}

\subsection{Two factor case}
\label{subsection.two factor}

In this section, we show that when $w$ has only two factors,
we can drop the assumption that $w\in S_{\hat{x}}$  for some $x\in[n]$.
We define crystal operators in the two factor case by reducing 
to the $w\in S_{\hat{x}}$ case and then proceeding as in Section~\ref{subsection.definition}.

Let $w\in \tilde{S}_n$ with $uv \in \mathcal W_{w,(\beta_1,\beta_2)}$.
Do the following initial bracket algorithm: Whenever $i$ is in $\content(u)$ and $i+1$ is in $\content(v)$, 
bracket them. Now either:
\begin{enumerate}
\item $i$ is in a block of the form
$(\ldots [b \cdots b-t] \ldots ) (\ldots [b \cdots b-t+1]\ldots )$ where $b\in \content(u)$ is unbracketed
under the initial bracketing; we assume $t$ to be maximal; or
\item $i$ is in a block of the form
$(\ldots [b-1 \cdots b-t]\ldots ) (\ldots [b \cdots b-t]\ldots )$ where 
$b-t\in \content(v)$ is unbracketed under the initial bracketing; we assume $t$ to be maximal; or
\item $i$ is in a block of the form
$(\ldots [b-1 \cdots b-t]\ldots) ( \ldots [b \cdots b-t+1]\ldots )$ with $b\not \in \content(u)$ and 
$b-t \not \in \content(v)$; again assume that $t$ is maximal.
\end{enumerate}

\begin{remark}
\label{remark.letters}
Note that in Case (1) above $b+1 \not \in \content(v)$ since otherwise $b$ in $\content(u)$ would be bracketed.
Similarly, in Case (2) $b-t-1 \not \in \content(u)$ since otherwise $b-t$ in $\content(v)$ would be bracketed.
\end{remark}

\begin{lemma}
\label{lemma.existence of i}
Let $w\in \tilde{S}_n$ and $uv \in \mathcal W_{w,(\beta_1,\beta_2)}$. 
Then either $w \in S_{\hat x}$ 
for some $x\in[n]$ or there exists an $i\in \content(uv)=[n]$ in Case (3) above.
\end{lemma}

\begin{proof}
If $w \in S_{\hat x}$, we are done. So assume that $w\not \in S_{\hat x}$ for any $x$.
Note that since $u$ is cyclically decreasing, there exists at least one letter $j\in[n]$ such that
$j \not \in \content(u)$. The same holds for $v$. Since all letters in $[n]$ appear in $\content(uv)$ 
neither $\content(u)$ nor $\content(v)$ can be empty. Hence there must be a letter $a\in \content(u)$ 
such that $a+1\not \in \content(u)$. Since all letters
appear in $\content(uv)$, we must have $a+1 \in \content(v)$. This implies that we have at least one initially bracketed 
letter in $\content(uv)$. Now assume by contradiction that all initially bracketed letters are in Cases (1) 
or (2) above.

If $\content(uv)$ contains a block $(\ldots [b \cdots b-t] \ldots ) (\ldots [b \cdots b-t+1]\ldots )$, then
$b+1\not \in \content(v)$ (else we are in Case (3) or $t$ is not maximal). Hence $b+1\in \content(u)$. 
Let $s$ be maximal such that
$b+j\in \content(u)$ for $1\le j\le s$, but $b+j \not \in \content(v)$. If $b+j=b-t$ (where recall that we take all letters 
$\mod n$), then all letters in $[n]$ occur in $\content(u)$, which is not possible. Hence another block
$(\ldots [b' \cdots b'-t'] \ldots ) (\ldots [b' \cdots b'-t'+1]\ldots )$ or 
$(\ldots [b'-1 \cdots b'-t']\ldots ) (\ldots [b' \cdots b'-t']\ldots )$ must occur.
If only blocks of the first form occur, then as in the previous argument all letters occur in $\content(u)$, which
is a contradiction. But note since $s>0$ and $b+s+1=b'-t'$, we have that $b'-t'-1\in \content(u)$, which by
Remark~\ref{remark.letters} means that we are not in Case (2), so we must be in Case (3), contradicting our
assumptions.

If we had started with a block of Case (2) initially, we would have arrived at a contradiction in similar fashion.

This proves that Case (3) must occur.
\end{proof}

By Lemma~\ref{lemma.existence of i} and its proof, there exists a $b$ such that $b \not \in \content(u)$,
$b\in \content(v)$ and $b-1$ is of Case (3). Remove the initially bracketed $(b-1,b)$-pair in $\content(uv)$.
Now it is not hard to check that all definitions and properties of the crystal operators on affine factorizations
of Sections~\ref{subsection.definition} and~\ref{subsection.properties} still go through with $x=b$
(and any braid or commutation relations still hold even with the $(b-1,b)$-pair present).
Hence we have crystal operators in the two factor case as well, even if $uv \not \in S_{\hat x}$ for any $x\in [n]$.

\begin{theorem}
\label{theorem.two factor}
For any $w\in \tilde S_n$ for which there is an affine factorization into two factors, $B(w)$ carries the 
structure of an $U_q(\mathfrak{sl}_2)$-crystal.
\end{theorem}

A common generalization of Theorems~\ref{theorem.crystal} 
and~\ref{theorem.two factor} (a ``crystal theorem") for 
more general $w\in \tilde S_n$ would be extremely interesting.
Since a generic affine Stanley symmetric function does not have
a nonnegative Schur expansion,
such a theorem will not exist without 
generalizing the notion of crystal.
However, there are large classes of affine permutations 
for which the expansion is Schur positive (modulo a 
natural ideal).  These cases would encode as highest weights
invariants tied to the WZW Verlinde fusion algebra and positroid
decompositions
(discussed further in Section~\ref{section.gromov witten}).

\section{Young--Specht modules}
\label{section.specht}

The crystal $B(w)$ for $w\in S_n$ corresponds to representations carrying an action of the symmetric 
group called {\it Young--Specht modules} (also called Specht modules in~\cite{ReS:1998}).  These modules $\mathcal S^D$ are associated to 
finite subsets $D$ of $\mathbb N\times \mathbb N$ called diagrams.  
Their origin was in Young's work~\cite{Young:all} to explicitly produce the irreducible representations 
of the symmetric group.  He required only Ferrers diagrams, the graphical depiction
of a partition $\lambda=(\lambda_1,\ldots,\lambda_m)$ with
non-increasing positive integer entries obtained by stacking
rows of $\lambda_i$ boxes in the left corner (with its smallest row at the top).
Here $\ell(\lambda):=m$ is called the length of the partition $\lambda$.
The set of Young--Specht modules indexed by Ferrers diagrams $\lambda$, 
where $|\lambda|=\sum_{i=1}^m \lambda_i = \ell$, is a complete set of irreducible $S_\ell$-modules.

It has since been established that other subclasses of Young--Specht modules 
are fundamental as well.  For example, Young--Specht modules indexed by
skew-shaped diagrams give $\mathfrak {sl}_{\ell}$-representations,
and their decomposition as a direct sum of irreducible submodules 
\begin{equation}
\label{skewspecht}
	\mathcal S^{\nu/\lambda} = \bigoplus_\mu c_{\lambda,\mu}^\nu \,\mathcal S^\mu
\end{equation}
yields multiplicities $c_{\lambda,\mu}^\nu$ that are 
given by the acclaimed \textit{Littlewood--Richardson (LR) rule} 
(details to follow).

Another notable family consists of the Young--Specht modules indexed by Rothe
diagrams of permutations, defined uniquely for each $w\in S_n$ to be
$$
	D(w) = \{(i, w(j)) \mid 1 \leq  i < j \leq  n, w(i) > w(j)\}\,.
$$
Our primary goal in this section is to provide the crystal for these.
We also discuss how a subcase of our construction yields a new 
characterization of the $\mathfrak{sl}_{\ell}$-crystal \cite{KN:1994},
and we give a number of results concerning the decomposition of Young--Specht modules 
into their irreducible components.

Before we begin, recall that the definition of $\mathcal S^D$ requires 
{\it fillings} $f$ of $D$, which are bijections 
$f\colon D\mapsto \{1,\ldots,\ell\}$ where $\ell=|D|$.  The Young--Specht module 
carries a natural left action of $S_\ell$ on fillings by the permutation
of entries.  The row group $R(f)$ of a filling $f$ is the subgroup of permutations
$\sigma\in S_\ell$ which act on $f$ by permuting entries 
within their row and similarly, the column group $C(f)$ is the subgroup 
that permutes entries within their columns.  The Young symmetrizer of a 
filling $f$ is given by
$$
	y_f = \sum_{p\in R(f)}\sum_{q\in C(f)} \sign(q) qp
	\in \mathbb C[S_\ell]\,.
$$
\begin{definition}
For each diagram $D$ and filling $f$,
the Young--Specht module $\mathcal S^D$ is the $S_\ell$-module
$\mathbb C[S_\ell]y_f$, where $\ell=|D|$.
\end{definition}

\subsection{Young--Specht modules and crystals for skew shapes}

A foundational example in crystal theory is the $\mathfrak{sl}_{\ell}$-crystal~\cite{LS:1978,LS:1979,KN:1994} 
on skew tableaux which, by Schur--Weyl duality, can be
associated to the Young--Specht modules $\mathcal S^{D}$ for skew shapes $D$.
Our point of departure is to recall the crystal on tableaux and to show that it is 
a special case of $B(w)$ on affine factorizations. 

The vertices of the $\mathfrak{sl}_{\ell}$-crystal 
crystal $B(\nu/\lambda)$ consist of the {\it semi-standard skew tableaux}
$\SSYT(\nu/\lambda)$ over the alphabet $\{1,2,\ldots,\ell\}$. Here
$t\in\SSYT(\nu/\lambda)$ when it is a filling of the diagram $D=\nu/\lambda$ 
with letters placed non-decreasing 
across rows and increasing up columns.  Its {\it weight} is defined by 
the composition $\mu=(\mu_1,\ldots,\mu_\ell)$ where $\mu_i$ records the
number of times $i$ occurs in $t$.  

Crystal operators $\et_i$ and $\ft_i$ for
$1\le i< \ell$ are defined on $t\in\SSYT(\nu/\lambda)$ using a bracketing of 
the letters $i$ and $i+1$ in $t$.
Scan the columns of $t$ from right to left,
bottom to top. When a letter $i+1$ appears, pair it with the closest previously
scanned $i$ in this scanning order that has not yet been paired (if possible).
Then $\ft_i(t)$ is the skew tableau obtained from $t$ by changing the 
rightmost unpaired $i$ into an $i+1$. If none exists, $\ft_i(t) = \bf{0}$.
Similarly, $\et_i(t)$ is obtained from $t$ by changing the 
leftmost unpaired $i+1$ into an $i$ and if none exists, $\et_i(t) = \bf{0}$.

\begin{example}
\label{example.ef on skew}
In the following skew tableau,
bracketed letters $2$ and $3$ are indicated in red
and the crystal operators $\et_2$ and $\ft_2$ act on the letter in the bold box.
$$
 \tableau[scY]{\color{red}{3} \cr \bl & 1 &\color{red}{2} & 2& \color{red}{3}
 \cr \bl&\bl&1&1 &\color{red}{2} &\tf 3& 3}
 \quad
 \begin{matrix}
 \rightarrow \et_2\rightarrow\\
 \leftarrow \ft_2 \leftarrow
 \end{matrix}
 \quad
 \tableau[scY]{\color{red}{3} \cr \bl &  1 &\color{red}{2}&  2&\color{red}{3}
 \cr \bl&\bl&1&1 &\color{red}{2} &\tf 2& 3}
 $$
\end{example}

Recall that Theorem~\ref{theorem.highestwts} indicates that highest weights correspond 
to irreducible components. In this setting,
$t\in B(\nu/\lambda)$ is highest weight if $\et_i t=\bf{0}$ for all $1\le i< \ell$ and 
thus the multiplicities in \eqref{skewspecht} are given by the following combinatorial objects.

\smallskip
\noindent \textit{Crystal version of Littlewood--Richardson Rule.} 

\noindent
$c_{\lambda,\mu}^\nu$ is the number of all semi-standard skew tableaux
$t\in \SSYT(\nu/\lambda,\mu)$ of shape $\nu/\lambda$ and weight $\mu$ such that $t$ is highest weight.
\smallskip

\noindent
Although the first rigorous proof did not appear until 1977 \cite{Schuetz:1977},
this rule was originally formulated in 1934~\cite{LR:1934} by counting 
objects called \textit{Yamanouchi tableaux}.  In particular, the 
Littlewood--Richardson rule specifies that $c_{\lambda,\mu}^\nu$ counts 
the number of semi-standard tableaux of shape $\nu/\lambda$ 
and weight $\mu$ with a Yamanouchi row word (that is, reading 
right to left and bottom to top, there are never more $i+1$'s than 
$i$'s, for all $i$).
It is not hard to see that the condition of being Yamanouchi is equivalent 
to being highest weight in the crystal. 

The crystal $B(w)$ introduced in Section~\ref{section.crystal operators} reduces to 
the crystal $B(\nu/\lambda)$ when $w\in S_n\subset\tilde S_n$ is {\it $321$-avoiding}
-- that is, when none of its reduced words contain a
braid $s_i s_{i+1} s_i$.  This subclass of permutations is in bijection~\cite{BJS:1993} 
(see also~\cite[Theorem~2.3.1(i)]{Garsia:2002})
with skew shapes fitting inside a rectangle by removing all rows and columns without cells
from the Rothe diagram $D(w)$.

\begin{prop}
If $w\in S_n$ is $321$-avoiding, let $D(w)=\nu/\lambda$ be the corresponding skew shape.
As $U_q(\mathfrak{sl}_{\ell})$-crystals, $B(\nu/\lambda)$ over the alphabet $\{1,2,\ldots, \ell\}$
is isomorphic to $B(w)$ with $\ell$ factors.
\end{prop}

\begin{proof}
Identify each cell $(i,j)$ in a skew semi-standard tableau $t\in \SSYT(\nu/\lambda,\mu)$
with a label $j-i+\ell(\nu)$. From $t$ we are going to produce an affine factorization $w^\mu$ of $w$
of weight $\mu$. For each letter $1\le r\le \ell$ in $t$, record the labels of all letters $r$ from right to left in $t$. 
This yields a decreasing word $w^r$ which is the $r$-th factor from the 
end in $w^\mu=w^{\ell} \cdots w^1$.
It is not hard to see that the bracketing rules for letters $r$ and $r+1$ in $t$ are equivalent to
the bracketing rules in factors $r$ and $r+1$ in $w^\mu$. 
Then $\et_r$ transforms the leftmost unbracketed letter $r+1$ in $t$ to an $r$.
This corresponds precisely to moving the rightmost unbracketed label from the $(r+1)$-th factor to the
$r$-th factor in $w^\mu$. Since $w$ does not contain any braids, the label moves unchanged between the factors.
\end{proof}

\begin{example}
We label each cell $(i,j)$ in the skew tableau of Example~\ref{example.ef on skew} by $j-i+\ell(\nu)$:
$$
 \tableau[scY]{\color{red}{3}_1 \cr \bl & 1_3 &\color{red}{2}_4 & 2_5& \color{red}{3}_6
 \cr \bl&\bl&1_5&1_6 &\color{red}{2}_7 &\tf 3_8& 3_9}
$$
so that the factorization is $w^\mu = (9\textbf{8}{\color{red}61})({\color{red}7}5{\color{red}4})(653)$. The colored
letters are the bracketed ones and the bold entry is the label moved by $\et_2$.
\end{example}

\subsection{Young--Specht modules and crystals for Rothe diagrams}
For any affine permutation $\tilde{w}\in S_{\hat x}\subset \tilde S_n$, 
the crystal $B(\tilde{w})$ on affine factorizations of $\tilde{w}$ gives the structure of 
more general Young--Specht modules.  
Since $S_{\hat x}$ for any $x\in [n]$ is isomorphic to $S_n$, we can
associate a permutation $\tau_x(\tilde{w})$ in $S_n$  to each $\tilde{w}\in S_{\hat x}$
by shifting the generators of $\tilde{w}$ by $-x\mod n$. 
We thus define the diagram of $\tilde{w}$ to be
$$
	\tilde D(\tilde{w}) := D(\tau_x(\tilde{w})) \,.
$$
The crystal $B(\tilde{w})$ on affine factorizations of $\tilde{w}\in S_{\hat x}$ 
gives the structure of the modules $\mathcal S^{\tilde D(\tilde{w})}$.  

\begin{theorem}
\label{theorem.specht}
For any $\tilde w\in S_{\hat x}\subset \tilde S_n$ with $x\in [n]$,
the decomposition of the Young--Specht module $\mathcal S^{\tilde D(\tilde w)}$ into
irreducible submodules is 
\begin{equation}
\mathcal S^{\tilde D(\tilde w)} = \bigoplus_\lambda a_{\tilde w,\lambda}\, \mathcal S^\lambda\,,
\end{equation}
where the multiplicity $a_{\tilde w,\lambda}$ is the
number of highest weight factorizations in $\mathcal W_{\tilde w,\lambda}$.
\end{theorem}

\begin{proof}
For $v\in S_n$, it can be deduced from results in~\cite{KP:2004,ReS:1995,ReS:1998}
that the Frobenius characteristic of $\mathcal S^{D(v)}$ 
is the Stanley symmetric function $F_v(x)$, where we recall these can
be viewed as a special case of the functions in \eqref{affstan}.
Therefore, if $\tilde w\in {S}_{\hat x}\subset \tilde S_n$ for some $x\in[n]$,  
we have that 
\begin{equation}
\label{Fchar}
\operatorname{char}(\mathcal S^{\tilde D(\tilde w)}) = F_{\tau_x(\tilde{w})} = F_{\tilde w}\,.
\end{equation}
By Theorem~\ref{theorem.crystal}, $B(\tilde w)$ is a $U_q(\mathfrak{sl}_{\ell})$-crystal 
in the category of integrable highest weight crystals. Hence by Theorem~\ref{theorem.highestwts},
the irreducible components are in one-to-one correspondence with highest weight vectors. 
Selecting the highest weight vectors of weight $\lambda$ yields the result.
\end{proof}

From the previous result and Theorem~\ref{theorem.two factor}, 
the statement can instead be interpreted on the level of 
symmetric functions by recalling that the Schur functions 
$s_\lambda$ are characters
of the irreducible Young modules $\mathcal S^\lambda$.

\begin{corollary}
\label{corollary.a}
For any $\tilde w\in S_{\hat x}\subset\tilde S_n$ with $x\in [n]$,
or for any $\tilde w\in \tilde S_n$ and $\ell(\lambda)\leq 2$, 
the coefficient $a_{\tilde w,\lambda}$ in
\begin{equation}
\label{equation.F schur expansion}
    F_{\tilde w}(x) = \sum_\lambda a_{\tilde w,\lambda}\, s_\lambda
\end{equation}
enumerates the highest weight factorizations in $\mathcal W_{\tilde w,\lambda}$.
\end{corollary}

\begin{example}
Example~\ref{example.crystal} shows that the crystal $B(w)$ of type $A_2$ for $w=s_3 s_4 s_1 s_2\in S_5$
has two highest weight vectors, one of weight $(2,1,1)$ and one of weight $(2,2)$,
matching the Schur expansion of the Stanley symmetric
function $F_{s_3s_4s_1s_2} = s_{(2,2)} + s_{(2,1,1)}$:
\begin{verbatim}
  sage: W = WeylGroup(['A',4],prefix='s')
  sage: w = W.from_reduced_word([3,4,1,2])
  sage: Sym = SymmetricFunctions(QQ)
  sage: s = Sym.schur()
  sage: s(w.stanley_symmetric_function())
  s[2, 1, 1] + s[2, 2]
\end{verbatim}
\end{example}

For $w\in S_n$, the coefficient $a_{w,\lambda}$ of a Schur function $s_\lambda$ 
in $F_w(x)$ was previously characterized by Fomin and 
Greene; they proved~\cite[Theorem 1.2]{FG:1998} that $a_{w,\lambda}$ 
counts the number of semi-standard tableaux of shape $\lambda'$ (the transpose 
of $\lambda$) whose column-reading word is a reduced word of $w$.  
Corollary~\ref{corollary.a}
thus implies such tableaux are in bijection with highest weight factorizations.

\begin{corollary}
\label{corollary.a bij}
For any permutation $\tilde w\in S_{\hat x}\subset \tilde S_{n}$ and partition $\lambda$ with $\ell(\lambda)\le \ell$,
the cardinality of the set
\[
	\{v^\ell \cdots v^1 \in \mathcal W_{\tilde w,\lambda} \mid \et_i (v^\ell \cdots v^1) = \bf{0} 
\text{ for all $1\le i<\ell$}\}
\]
of highest weight factorizations equals the number of
semi-standard tableaux of shape $\lambda'$ 
whose column-reading word is a reduced word of $\tilde w$.
\end{corollary}

As we will show in Theorem~\ref{theorem.intertwine}, this result can be proved bijectively
by extending the Edelman--Greene (EG) correspondence~\cite{EG:1987} between 
reduced words for $w\in S_n$ and pairs of certain same-shaped row 
and column increasing tableaux $(P,Q)$.  In fact, the bijection applies 
to the full crystal rather than just highest weight elements.

The basic operation needed for the EG-correspondence is a
variant of RSK-insertion.  Namely, the EG-insertion of letter 
$a$ into row $r$ of a tableau is defined by picking out the 
smallest letter $b>a$ in row $r$. If no such $b$ exists, the letter $a$ is placed
at the end of row $r$.
If $b=a+1$ and $a$ is also contained in row $r$, then $a+1$ is inserted 
into row $r+1$.  Otherwise, $b$ is replaced by $a$, and $b$ is inserted 
into row $r+1$.  In the last two cases, we say that $b$ has been bumped.  
This given, an insertion tableau $P$ and recording tableau $Q$
are constructed from a reduced word $w_\ell\cdots w_2 w_1$ 
starting from $P^0=Q^0=\emptyset$ and
iteratively defining $P^i$ by inserting $w_i$ into the bottom row 
of $P^{i-1}$.  Letters are bumped until a letter $a$ is to be 
inserted into a row $r$ containing no letter larger than $a$,
at which point $a$ is put at the end of row $r$.
$Q^{i}$ is then defined by adding $i$ to the end 
of row $r$ in $Q^{i-1}$. Finally, $P=P^{\ell}$ and $Q=Q^{\ell}$.

\begin{theorem}\cite{EG:1987}
\label{theorem:eg}
Each reduced word for $w\in S_n$ corresponds to a unique pair of tableaux
$(P,Q)$ of the same shape, where the column reading of the transpose of $P$
is a reduced expression for $w$ and $Q$ is standard.
\end{theorem}

For $\tilde w\in S_{\hat x}$, we more generally define a map on $\mathcal W_{\tilde w,\alpha}$
where
$$
\varphi_{\operatorname{EG}} \colon v^\ell \cdots v^1 \mapsto (P,Q)\,,
$$
for an appropriate pair of tableaux $(P,Q)$ with $Q$ now semi-standard of weight $\alpha=(\alpha_1,\ldots,\alpha_\ell)$.
In particular, let $P^0=Q^0=\emptyset$ and define $P^i$, for $i=1,\ldots,\ell$, 
by EG-inserting into $P^{i-1}$ the word $w_{\alpha_i} \cdots w_1$, where
$\content(v^i)=\{w_1,\ldots,w_{\alpha_i}\}$ and $w_{\alpha_i}>\cdots>w_1$ under the order of \eqref{orderx}.
$Q^i$ is defined by adding letter $i$ to $Q^{i-1}$ in cells
given by $\shape(P^{i})/\shape(P^{i-1})$.
Set $\varphi_{\operatorname{EG}}^Q(v^\ell \cdots v^1)=Q$.

\begin{remark}
\label{remark:eg}
EG-insertion enjoys many of the same properties as RSK-insertion.
For example, given that cell $c_y$ is added to a tableau when $y$ 
is EG-inserted, and cell $c_x$ is added when $x$ is then EG-inserted 
into the result, $c_x$ lies strictly east of $c_y$ when $x>y$, and 
$c_x$ lies strictly higher than $c_y$ when $x<y$. 
\end{remark}

\begin{theorem}
\label{theorem.intertwine}
For any $\tilde w\in S_{\hat x}\subset \tilde S_{n}$, 
the map $\varphi_{\operatorname{EG}}^Q$ is a crystal isomorphism
 \[
	B(\tilde w)\cong \bigoplus_\lambda B(\lambda)^{\oplus a_{\tilde w,\lambda}}\;.
\]
In particular,
\[
	\varphi_{\operatorname{EG}}^Q \circ \et_i = \et_i \circ \varphi_{\operatorname{EG}}^Q 
	\qquad \text{and} \qquad
	\varphi_{\operatorname{EG}}^Q \circ \ft_i = \ft_i \circ \varphi_{\operatorname{EG}}^Q.
\]
\end{theorem}

\begin{proof}
Fix $\tilde w\in S_{\hat x}$ for some $x\in[n]$.
We first note that $\varphi_{\operatorname{EG}}$ is a bijection
between $\mathcal W_{\tilde w,\alpha}$ and the set of pairs of same-shaped
tableaux $(P,Q)$ where the column-reading word of the transpose of $P$ is a reduced
expression for $\tilde w$ and $Q$ is semi-standard of weight $\alpha$.
That is, given $v^\ell\cdots v^1\in \mathcal W_{\tilde w,\alpha}$,
let $(P,Q)=\varphi_{\operatorname{EG}}(v^\ell\cdots v^1)$ and
recall that $P=P^\ell$ where $P^\ell$ is defined by 
inserting the (distinct) letters of $\content(v^{\ell})$ from smallest to 
largest into $P^{\ell-1}$. By Remark~\ref{remark:eg},
$Q^\ell/Q^{\ell-1}$ is a horizontal $\ell(v^\ell)$-strip
and we iteratively find $Q$ to be semi-standard of weight $\alpha$. 
The column reading word of the transpose of $P$ is a reduced expression for $\tilde w$
by Theorem~\ref{theorem:eg}.  It is not difficult to see that the process 
is invertible  by reverse EG-bumping letters from $P^i$ that lie in the 
positions determined by cells of $\shape(Q^i)/\shape(Q^{i-1})$ taken from 
right to left.

Let us now denote the letters  in $\content(v^{i+1})$ by $y_{\alpha_{i+1}} \cdots y_1$ and the letters
in $\content(v^i)$ by $x_{\alpha_i} \cdots x_1$ in the order in~\eqref{orderx}.
Let $a=x_{j}$ be the leftmost unbracketed letter in the pairing in Section~\ref{subsection.definition}.
Inserting the letters $x_1,\ldots,x_{\alpha_i}$ under the EG-insertion yields $\alpha_i$ insertion paths that move
strictly to the right in the tableaux $P^i$ by Remark~\ref{remark:eg}. Since $a=x_{j}$ is the leftmost unbracketed  letter in 
$\content(v^i)$, by Lemma~\ref{lem:Ruvdecomp} there exists an index $1\le m \le \alpha_{i+1}$ such that
$x_j<y_m<y_{m+1}< \cdots < y_{\alpha_{i+1}}$ and $y_1<y_2 < \cdots < y_{m-1}<x_{j-1}$ in the order~\eqref{orderx}.
In addition, all letters $y_1,\ldots, y_{m-1}$ are bracketed under the crystal bracketing which means that
the insertion paths of these letters are weakly to the left of the insertion path of $x_{j-1}$ and no letter can bump $x_j$.
Also, the letters $x_{j+1},\ldots,x_{\alpha_i}$ are bracketed under the crystal bracketing so that of the letters $i$ in 
$Q^i$ corresponding to the insertion paths of $x_j,\ldots,x_{\alpha_i}$ precisely
one is not bracketed with an $i+1$ in $Q^{i+1}$.

Now under $\ft_i$ the letter $a=x_j$ moves from $\content(v^i)$ to the letter $a+s$ in 
$\content(v^{i+1})$. As a result, the insertion paths of $x_{j+1},\ldots, x_{\alpha_i}$ either stay (partially)
in their old track or move left (partially) to the insertion of the previously inserted letter.
Similarly, the insertion paths of the corresponding $y_h$ move (partially) left. The new letter $a+s$
in $v^{i+1}$ after the application of $\ft_i$, then causes the previously unpaired letter $i$ in $Q^{i+1}$
in the insertion to become an $i+1$, possibly by shifting the insertion paths of the subsequent 
$y_h$ to the right. This proves the claim for $\ft_i$.

The proof for $\et_i$ is similar.
\end{proof}

A by-product of Theorem~\ref{theorem.intertwine} is a bijective proof of
Corollary~\ref{corollary.a bij}, where
the tableau associated to a highest weight element
$v^\ell \cdots v^1$ is the transpose of the insertion tableau 
$\varphi_{\operatorname{EG}}^P(v^\ell \cdots v^1)=P$.
It further gives a crystal theoretic analogue of the relation between the 
Edelman--Greene insertion of a reduced word of $w\in S_n$ and 
the RSK insertion of its peelable word given in~\cite{ReS:1995}.

\begin{example}
Given the highest weight factorization $v^3v^2v^1=(1)(2)(32)$,
with weight $\lambda=(2,1,1)$ of the permutation $s_1s_2s_3s_2\in S_4$,
the successive insertion of $v^i$, for $i=1,2,3$ yields
\[
	\bigl( \;{\footnotesize\tableau[scY]{2&3}\;,\; \tableau[scY]{1&1}} \;\bigr) \qquad
\left(\;{\footnotesize\tableau[scY]{3\cr 2&3}\;,\; \tableau[scY]{2\cr 1&1}}\;\right)\qquad
\left( \;{\footnotesize\tableau[scY]{3\cr 2\cr 1&3} \;,\; \tableau[scY]{3\cr 2\cr 1&1}}
\;\right)\;=\; (P,Q)\;.
\]
The column-reading word of the transpose of $P$ is $3123$,
indeed a reduced word for $s_1s_2s_3s_2$, demonstrating the bijective correspondence 
of Corollary~\ref{corollary.a bij}.
\end{example}

Another immediate outcome of our crystal $B(w)$ is Stanley's famous result~\cite{Stanley:1984} that the
number of reduced expressions for the longest element $w_0 \in S_n$ is equal to the number of standard 
tableaux of staircase shape $\rho=(n-1,n-2,\ldots,1)$. Namely, in $B(w_0)$ there is only one highest weight element
given by the factorization $(s_1)(s_2 s_1) (s_3 s_2 s_1) \cdots (s_{n-1} s_{n-2} \cdots s_1)$. Hence
$B(w_0)$ is isomorphic to the highest weight crystal $B(\rho)$. The reduced words of $w_0$ are precisely
given by the factorizations of weight $(1,1,\ldots,1)$. In $B(\rho)$ they are the standard tableaux of shape $\rho$.
The bijection between the reduced words of $w_0$ and standard tableaux of shape $\rho$ induced by the
crystal isomorphism is precisely $\varphi_{\operatorname{EG}}^Q$ (which due to our conventions of reading
the factorization from right to left gives the transpose of the standard tableau from the straight EG-insertion).

\section{Highest weights and geometric invariants}
\label{section.gromov witten}

Here we study families of constants including Gromov--Witten invariants 
for flag varieties (and in particular, Schubert polynomial 
structure constants~\eqref{schubcon}), the structure constants 
for the Verlinde (fusion) algebra of the Wess--Zumino--Witten model,
and the decomposition of positroid classes into Schubert classes.
Our approach is to apply the $B(\tilde w)$-crystal introduced 
in Section~\ref{section.crystal operators} to affine Schubert calculus.

To be precise, as discussed in the introduction, $H_*(\Gr)$ is isomorphic to the subring
$$
\Lambda_{(n)}=\mathbb Z[h_1,\ldots,h_{n-1}]\,
$$
of the ring of symmetric functions $\Lambda$, where $h_r=\sum_{i_1\leq\cdots \leq i_r} 
x_{i_1}\cdots x_{i_r}$.  
Representatives for the Schubert homology classes are given by 
a basis for $\Lambda_{(n)}$ made up of symmetric functions
called {\it $k$-Schur functions} (hereafter, $k=n-1$).  
Denoted by $s_{\tilde u}^{(k)}$, these are indexed by 
$\tilde u\in \tilde S_n^0$.  The importance of this basis to our study
is that the Schubert structure constants for
$H_*(\Gr)$ match the coefficients in
\begin{equation}
\label{klrbody}
s_{\tilde u}^{(k)}\,s_{\tilde w}^{(k)} = \sum_{\tilde v\in \tS} c_{\tilde u,\tilde w}^{\tilde v,k} \, s_{\tilde v}^{(k)}\,,
\end{equation}
and the families of constants under study here arise as
subsets of these {\it affine Littlewood--Richardson coefficients}.

After recalling the definition of $k$-Schur functions, we begin
by proving that the affine Stanley symmetric functions are none 
other than functions that arose by studying duals of the
$k$-Schur functions.
In doing so, we can apply the crystal $B(\tilde w)$ to the study of affine 
LR numbers and subsequently, to the study of the aforementioned constants.
In this section, to avoid confusion, we use the convention that affine permutations 
are denoted by $\tilde w$ and usual permutations 
of $S_n$ appear as $w$ -- without the tilde.

\subsection{The affine Stanley/dual $k$-Schur correspondence}
\label{subsection.kschur}

There are many equivalent formulations for the $k$-Schur basis.
In the spirit of our presentation, we use its characterization
in terms of the homogeneous basis $\{h_\lambda\}_{\lambda\in\mathcal P^n}$,
where $h_\lambda=h_{\lambda_1}\cdots h_{\lambda_\ell}$
and $\mathcal P^n=\{\lambda \mid \lambda_1<n\}$ is the set of
partitions with all parts shorter than $n$.
This expansion relies on the matrix $\mathcal K$ whose entries,
$$
\mathcal K_{\tilde w,\mu}=|\mathcal W_{\tilde w,\mu}|\,,
$$ 
enumerate affine factorizations of $\tilde w\in\tS$ with fixed 
weight $\mu\in\mathcal P^n$.  

The matrix $\mathcal K$ is square since 
the set of affine Grassmannian elements $\tilde S_n^0$ is in bijection with $\mathcal P^n$.  
Namely, since the window of an affine Grassmannian
$\tilde w\in \tS$ is increasing, $\lambda=\linv(\tilde w)$ is weakly decreasing and its last
entry is zero.  Thus, taking the transpose partition $\lambda'$, the map
\begin{equation}
\label{equation.LC}
	\LC\colon \tilde w\mapsto \linv(\tilde w)'
\end{equation}
sends $\tS\to \mathcal P^n$ and in fact, it is bijective~\cite{BB:1996}.
We use $\tilde w_\lambda$ to denote the inverse image of $\lambda \in \mathcal P^n$.

\begin{remark}
\label{remark:lcr}
The map $\LC$ is well-defined on the set of all
extended affine Grassmannian permutations and,
for fixed $r$, gives a 
bijection between the affine Grassmannian elements 
of $\tilde S_{n,r}$ and $\mathcal P^n$.  In particular,
$\LC(\tilde v) =\lambda$ for $\tilde v=\tau^r \tilde w_\lambda$.
\end{remark}

\begin{example}
For $\tilde{w}=[-2,0,1,4,12]\in\tilde S_5^0$,
the left inversion vector is $\linv(\tilde w)=(3,2,2,1,0)$
and its conjugate is $\LC(\tilde w)=(4,3,1)\in\mathcal P^5$.
Thus, in our notation, $\tilde w=\tilde w_{(4,3,1)}$.
\end{example}

\noindent
The matrix $\mathcal K$ is also unitriangular and thus
characterizes functions, for each $w\in\tS$, by
\begin{equation}
\label{equation.h kschur}
	s_{\tilde w}^{(k)} = \sum_{\mu\in\mathcal P^n}
	\overline{\mathcal K}_{\mu, \tilde w}
	\,
	h_\mu
	\,,
\end{equation}
where $\overline{\mathcal K}=\mathcal K^{-1}$. 
The set of these functions defines the {\it $k$-Schur basis} for $\Lambda_{(n)}$.

The ring $\Lambda_{(n)}$ is naturally Hopf dual to the quotient 
$\Lambda^{(n)}=\Lambda/\langle m_\lambda \mid \lambda_1\geq n\rangle$,
where $m_\lambda = \sum x_{\alpha_1}^{\lambda_1} 
x_{\alpha_2}^{\lambda_2}\cdots x_{\alpha_\ell}^{\lambda_\ell}$
over tuples $(\alpha_1,\ldots,\alpha_\ell) \in \mathbb N^\ell$ with distinct entries.
The elements $\{m_\lambda\}_{\lambda\in\mathcal P^n}$ may be chosen as 
representatives of the dual algebra.  Duality can be used to produce a second basis, 
now for the algebra $\Lambda^{(n)}$.  
Namely, a basis $\{\mathcal F_\lambda\}_{\la\in\mathcal P^n}$
can be characterized as the unique set of elements
in the subspace $\Lambda^{(n)}$ that are dual to $\{s_{\tilde u}^{(k)}\}_{\tilde u\in\tS}$
by appealing to the pairing 
\begin{equation} 
\label{hmdual}
	\langle ~\cdot~, ~\cdot~ \rangle\; : \; \Lambda_{(n)} \times \Lambda^{(n)} \rightarrow \mathbb Q\;,
\end{equation}
where $h_\mu \in \Lambda_{(n)}$ and $m_\la \in \Lambda^{(n)}$ are dual elements.
That is, $\left< h_\mu, m_\la \right> = \delta_{\la\mu}$.

These elements were first studied as a special case of
the {\it dual $k$-Schur functions}, a family defined 
for any skew diagram of 
$\mathcal D = \{\nu/\lambda \mid \ell(\tilde w_\nu \tilde w_\lambda^{-1})=\ell(\tilde w_\nu)-\ell(\tilde w_\lambda)\}$
by
\begin{equation}
\label{dualkschur}
\mathcal F_{\nu/\lambda} = \sum_{\mu \in\mathcal P^n}
\mathcal K_{\tilde w_\nu \tilde w_\lambda^{-1},\mu} \,m_\mu\,.
\end{equation}
The original motivation for their study was to 
produce affine LR (and WZW-fusion) coefficients.  
\begin{theorem}\cite{LM:2008}[Theorem 28]
\label{skewdualindual}
For any $\lambda\subset\nu\in\mathcal P^n$,
\begin{equation}
\mathcal F_{\nu/\lambda} = \sum_{\mu \in\mathcal P^n}
 c_{\tilde w_\lambda,\tilde w_\mu}^{\tilde w_\nu,k}\,
\mathcal F_\mu\,.
\end{equation}
\end{theorem}
Lam then introduced affine Stanley functions $F_{\tilde{w}}$ 
for $\tilde{w} \in \tilde S_n$ in \cite{Lam:2006} and proved that
\begin{equation}
\label{equation.skew dual k Schur}
	\mathcal F_{\nu/\lambda} = F_{\tilde w_\nu \tilde w_\lambda^{-1}}\,.
\end{equation}
The converse was not readily apparent at the time; 
it was believed that affine Stanley symmetric functions were more
general than dual $k$-Schur functions.  As it turns out, we discovered 
that every affine element of $\tilde S_n$ can be associated
to a skew diagram in such a way that the corresponding
affine Stanley symmetric function is a skew dual $k$-Schur 
function.\footnote{Thomas Lam mentioned 
that he knows a different (unpublished) proof of this fact.}

For this, it is convenient to work not only with $\tilde S_{n}=\tilde S_{n,0}$, 
but with the extended affine symmetric group
as defined in Section~\ref{subsection.affine permutations}.
Our correspondence hinges on an injection from permutations of $S_n$ into 
$\tilde S_{n,\binom{n}{2}}$ defined by
$$
\mathfrak {af}:
w\mapsto 
[w(1),w(2)+n,w(3)+2n,\ldots,w(n)+(n-1)n]\,.
$$
This map was introduced in \cite{LM:2013} as the crux of an association 
between Gromov--Witten invariants for flag manifolds and affine LR-coefficients
(further discussed in Section~\ref{gw}).  We shall need two properties 
of the interplay between $\mathfrak {af}$ and 
length that come out of a close examination of $\linv$ 
(e.g.~\cite{LM:2013}):
For $\tilde v\in\tS$ and $u\in S_n$,
\begin{equation}
\label{equation:lengthaf}
\ell(\tilde v\mathfrak{af}(\id))=\ell(\tilde v)+\ell(\mathfrak{af}(\id))\qquad
\text{and}\qquad
\ell(\mathfrak{af}(u))=\ell(\mathfrak{af}(\id))-\ell(u)\,.
\end{equation}

In fact, the left inversion vector of $\tilde v\mathfrak{af}(\id)$ and
of $\mathfrak{af}(u)$ is a partition.  Precisely, affine permutations in 
the image of $\mathfrak {af}$ always have increasing windows.
Moreover, when $\tilde v =[\tilde v(1)<\cdots <\tilde v(n)]$ is affine 
Grassmannian, 
$\tilde v \mathfrak {af}(\id)=[\tilde v(1),\tilde v(2)+n,\ldots,\tilde v(n)+(n-1)n]$
has an increasing window as well.  This allows for 
a well-defined action of $\LC$ on the range of $\mathfrak{af}$ as
well as on the product of elements of the form $\tilde v \mathfrak {af}(\id)$
for $\tilde v\in\tS$.

\begin{lemma}\cite{LM:2013} 
\label{lemma:permpartbij}
The composition $\LC\circ \mathfrak{af}$ is a bijection from $S_n$ onto
the set of partitions
$\{\lambda\mid \square_{k-1}\subseteq\lambda\subseteq\square_k\}$,
where $\square_k=(k,k-1,k-1,\ldots,1^k)$ is the concatenation 
of all $k$-rectangles.  Further, 
for any $\nu\in\mathcal P^n$ containing $\square_k$,
there is a unique $\tilde v\in \tS$ such that
$\LC(\tilde v\mathfrak{af}(\id))=\nu$.
\end{lemma}

\begin{definition}
\label{skewdualgrass}
For $\tilde w \in \tilde S_{n}$, consider the decomposition 
$\tilde w = \tilde v u$ where $u\in S_n$ is the permutation 
that rearranges the window of $\tilde w$ 
into increasing order and $\tilde v \in \tS$ is the resulting 
affine Grassmannian element.
Define
\begin{equation}
\kappa \colon \tilde w \mapsto \nu/\lambda\,,
\end{equation}
for $\nu=\LC(\tilde v\mathfrak {af}(\id))$ and $\lambda = 
\LC(\mathfrak {af}(u^{-1}))$.
\end{definition}

\begin{prop}
\label{skewbijection}
The map $\kappa$ is a bijection
\[
	\kappa \colon \tilde S_{n} \to \{\nu/\lambda \mid 
	\nu, \lambda \in \mathcal P^n \text{ and } \square_{k-1} \subseteq \lambda \subseteq \square_k \subseteq \nu \}\;.
\]
Moreover, if $\kappa(\tilde w) = \nu/\lambda$,
then $\ell(\tilde w)=|\nu|-|\lambda|$ and
$\tilde{w}=\tau^{\binom{n}{2}}\tilde w_{\nu} \tilde w_{\lambda}^{-1}
\tau^{-\binom{n}{2}}$.
\end{prop}

\begin{proof}
Given $\tilde \tau,\tilde w\in \tilde S_n$, consider
their decompositions $\tilde \tau=\tilde b \sigma$
and $\tilde w=\tilde v u$ where $\sigma,u\in S_n$ and
$\tilde b,\tilde v\in \tS$ according to Definition~\ref{skewdualgrass}.
If $\kappa(\tilde w)=\kappa(\tilde \tau)$, then
$\nu=\LC(\tilde b\mathfrak{af}(\id))
=\LC(\tilde v\mathfrak{af}(\id))$ and
$\lambda=\LC(\mathfrak{af}(\sigma^{-1})) = \LC(\mathfrak{af}(u^{-1}))$.
Since $\LC$ is a bijection between $\tilde S_{n,\binom{n}{2}}$ and
$\mathcal P^n$ , we have that
$\tilde b\mathfrak{af}(\id) =\tilde v\mathfrak{af}(\id)$ and
$\mathfrak{af}(\sigma^{-1}) = \mathfrak{af}(u^{-1})$ implying that
$\kappa$ is injective.
Moreover, for a generic $\tilde w=\tilde v u$, 
Lemma~\ref{lemma:permpartbij} indicates that
$\square_{k-1}\subseteq \LC(\mathfrak {af}(u^{-1}))
\subseteq 
\square_k\subseteq\LC(\tilde v\mathfrak {af}(\id))$
and thus the correct codomain has been specified.

To complete the proof that $\kappa$ is bijective, we construct the
preimage of $\nu/\lambda$ assuming that 
$\square_{k-1} \subseteq \lambda \subseteq \square_k \subseteq \nu$
and $\nu,\lambda \in \mathcal P^n$.
In particular, we take $\tilde w=\tilde v u$ for the unique 
$u\in S_n$ and $\tilde v\in \tS$ such that
$\nu=\LC(\tilde v\mathfrak{af}(\id))$ and
$\lambda=\LC(\mathfrak{af}(u^{-1}))$ that is guaranteed to 
exist by Lemma~\ref{lemma:permpartbij}.

Now consider $\kappa(\tilde w)=\nu/\lambda$ and the
usual decomposition $\tilde w=\tilde v u$.
Note that $\ell(\tilde w)=\ell(\tilde v)+\ell(u)$.
On the other hand, $\nu=\LC(\tilde v\mathfrak{af}(\id))$ and
$\lambda=\LC(\mathfrak{af}(u^{-1}))$ imply that
$\ell(\tilde v\mathfrak{af}(\id))=|\nu|$ and
$\ell(\mathfrak{af}(u^{-1}))=|\lambda|$.
From \eqref{equation:lengthaf}, we have that
$\ell(\tilde v\mathfrak{af}(\id))=\ell(\tilde v)+\ell(\mathfrak {af}(\id))$
and $\ell(\mathfrak {af}(u^{-1}))= \ell(\mathfrak {af}(\id))- \ell(u)$. 
Therefore, $|\nu|-|\lambda|=\ell(\tilde v)+\ell(u)=\ell(\tilde w)$.
Lastly, 
$$
\tilde{w}= \tilde{v} u = 
 \left(\tilde v\mathfrak {af}(\id)\right)\left(\mathfrak {af}(u^{-1})\right)^{-1}
=
\left(\tau^{\binom{n}{2}}\tilde w_\nu\right)
\left(\tau^{\binom{n}{2}}\tilde w_\lambda\right)^{-1}
$$
since $\tilde v\mathfrak{af}(\id),\mathfrak{af}(u^{-1})\in
\tilde S_{n,\binom{n}{2}}$ and $\nu=\LC(\tilde v\mathfrak{af}(\id))$ and
$\lambda=\LC(\mathfrak{af}(u^{-1}))$.
\end{proof}

\begin{corollary}
\label{afs2dks}
For every $\tilde w\in\tilde S_n$, the affine Stanley symmetric function
$F_{\tilde w}$ is the dual $k$-Schur function $\mathcal F_{\kappa(\tilde w)}$.
\end{corollary}

\begin{proof}
For any $\tilde w\in \tilde S_n$, if
$\nu/\lambda=\kappa(\tilde{w})$ then
$\tilde w = \tau^{\binom{n}{2}}\tilde{w}_\nu \tilde{w}_\lambda^{-1}
\tau^{-\binom{n}{2}}$ by Proposition~\ref{skewbijection}.
Since $\tau^r$ acts on affine elements by
a cyclic shift of the generators, the number of affine factorizations 
of $\tilde w$ with weight $\mu$ is the same as the
number of affine factorizations of $\tilde w_\nu\tilde w_\lambda^{-1}$ 
with weight $\mu$.  
By the definition of dual $k$-Schur and affine Stanley functions,
$F_{\tilde w}=\mathcal F_{\nu/\lambda}$.
\end{proof}

\begin{example}
For $\tilde{w}=s_3 s_0 s_2 s_3 = [-1,4,5,2] \in \tilde{S}_4$, we find
$\tilde{v}=[-1,2,4,5]$ and $u=s_2 s_3=[1,3,4,2]$. From this,
$\mathfrak{af}(u^{-1})=[1,8,10,15]$ and since $\mathfrak{af}(\id)=[1,6,11,16]$, 
we have $\tilde{v} \mathfrak{af}(\id)=[-1,6,12,17]$.
Thus, $\nu=\LC([-1,6,12,17])=(3,2,2,1^5)$ and $\lambda=\LC([1,8,10,15])=(3,1^5)$
implying $\tilde{w}_\nu \tilde{w}_\lambda^{-1} = s_1 s_2 s_0 s_1$.  This indeed is
$\tilde{w}=s_3 s_0 s_2 s_3$ up to a cyclic shift by 2 in the generators. 
\end{example}

\subsection{Structure constants for $H_*(\Gr)$}
\label{subsection.affine LR}

We are now in a position to connect the crystal $B(\tilde w)$
with homology of the affine Grassmannian $\Gr$.
We prove that the highest weights of $B(\tilde w)$ 
are affine LR coefficients which in turn are the Schubert 
structure constants of $H_*(\Gr)$.

The crystal applies to a subclass of the $k$-Schur structure constants
that appear in products of a $k$-Schur function with a Schur function 
indexed by $\mu\subset (r^{n-r})$, for any $1\leq r<n$.
These indeed are affine LR-coefficients of~\eqref{klrbody} by 
the $k$-Schur property that, for any $\mu\subset (r^{n-r})$,
\begin{equation}
\label{kschurisschur}
	s_{\tilde w_\mu}^{(k)}=s_\mu\,.
\end{equation}

We first revisit the affine Stanley symmetric functions with this restriction 
in hand.  While Theorem~\ref{skewdualindual} explains that the dual $k$-Schur 
expansion of $F_{\tilde w}$ yields positive coefficients, it does not 
suggest the same about the Schur expansion coefficients.  In fact, many 
of the Schur coefficients are not positive.
However, when $\mu\subset(r^{n-r})$, we shall prove that they are positive 
and in particular include Gromov--Witten invariants for full flags.

\begin{prop}
\label{af2klr}
For any $\tilde w\in \tilde S_n$, consider 
the Schur expansion 
\begin{equation}
\label{Fins}
F_{\tilde w} = \sum_{\mu} a_{\tilde w,\mu} \,s_\mu\,.
\end{equation}
For $\mu\subset (r^{n-r})$ with $1\leq r<n$,
the coefficient $a_{\tilde w,\mu}$ 
is the (non-negative) affine LR-coefficient
$c_{\tilde w_\lambda,\tilde w_\mu}^{\tilde w_\nu,k}$,
where $\kappa(\tilde w)=\nu/\lambda$.
\end{prop}
\begin{proof}
Let $\nu/\lambda = \kappa(\tilde w)$. 
By Theorem~\ref{skewdualindual} and Corollary~\ref{afs2dks},
we have
\[
	F_{\tilde w} = \mathcal F_{\nu/\lambda}
	= \sum_{\mu \in\mathcal P^n} c_{\tilde w_\lambda,\tilde w_\mu}^{\tilde w_\nu,k}\,
	\mathcal F_\mu\,.
\]
Since $\mathcal F_\mu$ and 
$s_\mu^{(k)}$ are dual under the Hall inner product, when $\mu\subset (r^{n-r})$
we use \eqref{kschurisschur} to find
\[
	c_{\tilde w_\lambda,\tilde w_\mu}^{\tilde w_\nu,k} = \langle F_{\tilde w}, s_\mu^{(k)} \rangle
	= \langle F_{\tilde w}, s_\mu \rangle = a_{\tilde w,\mu}
\,,
\]
and the claim follows.
\end{proof}

Having identified particular Schur coefficients in the affine Stanley 
(dual $k$-Schur) function as affine LR coefficients, we 
use highest weights in the crystal to describe them.
For the greatest generality of results, we appeal to another
property of $k$-Schur functions.  It relies on a
family of operators $\{R_i\}_{1\leq i<n}$ that act 
by changing the window $[\tilde u_1,\ldots,\tilde u_n]$ of 
an element $\tilde u\in \tilde S_{n,r}$ to
\begin{equation}
\label{Ri}
	R_i\tilde u=[\tilde u_1-(n-i),\tilde u_2-(n-i),\ldots,\tilde u_i-(n-i), \tilde u_{i+1}+i,\ldots,\tilde u_n+i]\,.
\end{equation}
Note that the application of various $R_i$ commutes.
It was shown in \cite{LM:2013} that when $\tilde u$ is affine
Grassmannian, 
\begin{equation}
\label{linvR}
\linv(R_r\tilde u)=\linv(\tilde u)+\sum_{i=1}^r(n-r)e_i\,,
\end{equation}
where the $e_i$ denote standard basis vectors.
Note then that $R_i(\id)=\tilde w_{(i^{n-i})}$
and hence the $R_i:=R_i(\id)$ are called {\it $k$-rectangles}.
In this case, the $k$-Schur function 
$s_{{R_i}}^{(k)}$ is simply $s_{(i^{n-i})}$ by \eqref{kschurisschur}.
We shall extensively use that the multiplication of a $k$-Schur function 
by such a
term is trivial.  By \cite[Theorem 40]{LM:2007},
\begin{equation}
\label{kschurrec}
s_{{R_i}}^{(k)}\,s_{\tilde u}^{(k)} = s_{R_i\tilde u}^{(k)}\,.
\end{equation}

\begin{theorem}
\label{theorem.main}
Consider $\tilde v,\tilde w\in \tS$ and $\mu\subset (r^{n-r})$ where $1\leq r<n$.
Let $R=\prod R_i$ be any product of $k$-rectangles.
If $\ell(\tilde v)-\ell(\tilde w)\neq |\mu|$,  
then $c_{R\tilde w_\mu,\tilde w}^{R\tilde v,k}=0\,.$
Otherwise, if either
(i) $\tilde v\tilde w^{-1}\in S_{\hat x}$ for some $x\in[n]$  or
(ii) $\ell(\mu)=2$, then
$$
c_{R\tilde w_\mu,\tilde w}^{R\tilde v,k}=
\# \text{ of highest weight factorizations of $\tilde v\tilde w^{-1}$ of weight $\mu$} \,.
$$
\end{theorem}

\begin{proof}
Given $\mu\subset (r^{n-r})$ for any $1\leq r<n$, first
consider the case when $R$ is the empty product.  For this,
we examine the coefficients in 
\begin{equation}
\label{spro}
s_{\tilde w_\mu}^{(k)} \,s_{\tilde w}^{(k)} = 
\sum_{\tilde v}\,c_{\tilde w_\mu,\tilde w}^{\tilde v,k}\, s_{\tilde v}^{(k)}\,.
\end{equation}
Degree conditions on polynomials imply that
$c_{\tilde w_\mu,\tilde w}^{\tilde v,k}=0$ unless $\ell(\tilde v)=\ell(\tilde w)+|\mu|$.
Otherwise, Proposition~\ref{af2klr} indicates that the coefficients are equivalent to
$a_{\tilde z,\mu}$, where $\tilde z=\tilde v\tilde w^{-1}$.
By Corollary~\ref{corollary.a},
this is the number of highest weight factorizations of $\tilde v\tilde w^{-1}$ 
of weight $\mu$.

Now multiply both sides of \eqref{spro} by $s^{(k)}_R$ for 
any product $R=\prod {R_i}$ and use Equation~\eqref{kschurrec} to find
that 
$s_{R\tilde w_\mu}^{(k)} \,s_{\tilde w}^{(k)} = \sum_{\tilde v}\,
c_{\tilde w_\mu,\tilde w}^{\tilde v,k}\, s_{R\tilde v}^{(k)}\,.$
Hence,
\begin{equation}
\label{redbyR}
c_{R\tilde w_\mu,\tilde w}^{R\tilde v,k}=
c_{\tilde w_\mu,\tilde w}^{\tilde v,k}\,,
\end{equation}
and the result thus holds in the generality stated. 
\end{proof}

\subsection{Flag Gromov--Witten invariants}
\label{gw}

Let $\Fl_n$ be the complete flag manifold (chains of vector spaces)
in $\mathbb C^n$.  $\Fl_n$ admits a cell decomposition into 
{\it Schubert cells}, indexed by permutations of $S_n$.
Details on their construction can be found in~\cite{Manivel:2001}.
The set of Schubert classes $\{\sigma_w\}_{w\in S_n}$ forms a
basis for the cohomology ring and counting points in the intersection 
of Schubert varieties (closures of Schubert cells) amounts 
to taking the cup product~\eqref{schubcon} of 
Schubert classes in $H^*(\Fl_n)$~\cite{BGG:1973,Dem:1974}.
In turn, Lascoux and Sch\"utzenberger introduced an explicit
set of polynomial representatives $\mathfrak S_w$
called Schubert polynomials whose structure coefficients,
\begin{equation}
\label{schubpro}
\mathfrak S_u \mathfrak S_v = \sum_{w} c_{u,v}^w\,\mathfrak
S_w\,,
\end{equation}
give these intersection numbers.

As a linear space, the {\it quantum cohomology} of $\Fl_n$
is $\QH^*(\Fl_n)= H^*(\Fl_n)\otimes\mathbb Z[q_1,\ldots,q_{n-1}]$
for parameters $q_1,\ldots, q_{n-1}$.  Its
intricacy is in the multiplicative structure, defined by
\begin{equation}
\label{qschubin}
\sigma_u\,*\,\sigma_w = \sum_v \sum_{\d} 
q_1^{d_1} \cdots q_{n-1}^{d_{n-1}}
\langle u,w,v\rangle_{\d}
\, \sigma_{w_0 v}\,,
\end{equation}
where the structure constants $\langle u,w,v\rangle_d$
are 3-point Gromov--Witten invariants of genus 0 which count
equivalence classes of certain rational curves in $\Fl_n$.
The combinatorial study of these invariants for flag manifolds 
is wide open.  In fact, a manifestly positive formula in the case 
that $q_1=\cdots=q_{n-1}=0$, when the invariants reduce to
coefficients in \eqref{schubpro}, has been an open problem for 40 years.

Our approach is to use an identification of the Gromov--Witten invariants 
with affine LR--coefficients that was made in \cite{LM:2013}.  It
requires the map $\mathfrak {af}: S_n\to \tilde S_{n,\binom{n}{2}}$, 
and we thus allow for $k$-Schur functions indexed 
by extended affine permutations with increasing window.
That is, using \eqref{equation.LC}, we set
$s_{\tilde w}^{(k)} = s_{\tilde w_{\LC(\tilde w)}}^{(k)}$.
Note that then affine LR coefficients may be
indexed by extended affine Grassmannian elements 
in $\tilde S_{n,r}$ as well.

\begin{theorem}(proven in \cite{LM:2013})
\label{theorem:gwklr}
For any $u,v,w\in S_n$, the degree $d$, 3-point Gromov--Witten 
invariants of genus zero for complete flags are
the affine LR--coefficients given by
\begin{equation}
\label{gwklr}
\langle u,w,w_0v\rangle_{\d} = c_{\mathfrak {af}(u),\mathfrak {af}(w)}^{
\mathfrak{af}_{\d}(v),k}\,,
\end{equation}
where 
\begin{equation}
\label{eq:afd}
\mathfrak{af}_{\d}(v)=\prod_{i=1}^{n-1}R_i^{d_{i-1}+d_{i+1}-2d_i+1}\,
\mathfrak {af}(v)
\,.
\end{equation}
If $\mathfrak{af}_{\d}(v)$ is not affine Grassmannian or 
$\ell(u)+\ell(w)=\ell(v)+2\sum_{i=1}^{n-1}d_i$, then
$\langle u,w,w_0v\rangle_{\d} = 0$.
\end{theorem}

Theorem~\ref{theorem.main} thus applies to the study of Gromov--Witten 
invariants. The imposed conditions translate to the
study of the natural subclass of $\langle u,w,w_0v\rangle_{\d}$ 
where $u$ is a coset representative of $S_n/S_r\times S_{n-r}$ for 
some $r$.  
A set of representatives is given by the {\it Grassmannian} 
permutations of $S_n$, characterized by having exactly one descent in position $r$.  
Each Grassmannian permutation $u$ can be identified with a partition 
$\lambda(u)=(\lambda_1,\ldots,\lambda_r)$ in the $r\times (n-r)$ 
rectangle by setting $\lambda_i=|\{u_j \mid u_{r+1-i}>u_j\text{ and } j>r\}|$.
We also use the complement partition $\lambda^\vee$ to $\lambda$ in this
rectangle; that is,
$\lambda^\vee=(n-r-\lambda_r,\ldots,n-r-\lambda_1)$. 

\begin{theorem}
\label{the:GW}
For any $\d\in\mathbb N^{n-1}$ and $u,w,v\in S_n$ where $u$ is Grassmannian 
with descent at position $r$, let $\mu' =\lambda(u)^\vee$.
If $\ell(v)\neq\ell(u)+\ell(w)-2\sum_i d_i$,
then $\langle u,w,w_0v\rangle_d=0$.
Otherwise, if (i) $(R\,v)w^{-1}\in S_{\hat x}$ for some $x\in [n]$ 
or (ii) $\ell(\mu)=2$, then
$$
\langle u,w,w_0v\rangle_{\d}=\# \;\text{of highest weight factorizations 
of $(R\,v)w^{-1}$
with weight $\mu$}
$$
for $R=R_r \, \prod_{i=1}^{n-1} R_i^{d_{i-1}+d_{i+1}-2d_i}$.
\end{theorem}

\begin{proof}
We shall use the correspondence $\langle u,w,w_0v\rangle_{\d}= 
c_{\mathfrak{af}(u),\mathfrak{af}(w)}^{\mathfrak{af}_{\d}(v),k}$
of \eqref{gwklr} as our guide.  We first examine $\mathfrak{af}(u)$ 
in the case that $u\in S_n$ has exactly one descent at position $r$.
It was shown in~\cite{LM:2013} that
\begin{equation}
\label{Rafu}
\linv(\tilde u)=\lambda(u)^\vee\,,
\end{equation}
for
$\tilde u = \prod_{i\neq r} R^{-1}_i \mathfrak{af}(u) \in\tilde S_{n,p}$,
where $p=\binom{n}{2}$.  
In particular, $\tilde u$ is affine Grassmannian and therefore
has a unique decomposition $\tilde u=\tau^p\tilde w_\mu$ for 
some $\mu\in \mathcal P^n$.
In fact, $\mu'=\lambda(u)^\vee$ since $\linv(\tilde u)=\linv(\tilde w_\mu)=\mu'$.
Therefore, $\mathfrak{af}(u)=\prod_{i\neq r} R_i\tau^p\tilde w_{\mu}$
for $\mu'=\lambda(u)^\vee$. 

Next consider 
$\mathfrak{af}_{\d}(v)= \left(\prod_{i\neq r}R_i \right)\, R\, \mathfrak {af}(v)$
where $R= R_r \prod_{i=1}^{n-1} R_i^{d_{i-1}+d_{i+1}-2d_i}$.
From \eqref{kschurrec}, it can be deduced that
\begin{equation}
\label{gwonedescent}
\langle u,w,w_0v\rangle_{\d}=c_{\prod_{i\neq r} R_i\tau^p\tilde w_{\mu}, 
\mathfrak{af}(w)}^{\prod_{i\neq r}R_i\, R\, \mathfrak {af}(v),k}
=
\begin{cases} 
c_{\tilde w_\mu,\tau^{-p}\mathfrak{af}(w)}^{\tilde w_\nu,k}
&\text{if } R\, \mathfrak {af}(v)= \tau^p\tilde w_\nu\;,\\
\\
0   &\text{otherwise.}
\end{cases}
\end{equation}
We thus assume that $R\, \mathfrak {af}(v)= \tau^p\tilde w_\nu$
and consider the quantity $q(u,w,v)=\ell(u)+\ell(w)-2\sum_id_i-\ell(v)
=\ell(\mathfrak{af}(\id))-\ell(v)-2\sum_i  d_i 
-\ell(\mathfrak{af}(\id))+\ell(w)+\ell(u)$.
Using that $\mu'=\lambda(u)^\vee$ and
the length conditions of~\eqref{equation:lengthaf}, we 
have that $q(u,w,v)=\ell(\mathfrak{af}(v))-2\sum_i  d_i +(n-r)r
-|\mu| -\ell(\mathfrak{af}(w))$.
By~\eqref{equation.length} and Property~\eqref{linvR}, we then have
$$
\ell(u)+\ell(w)-\ell(v)-2\sum_id_i
=\ell(R\mathfrak{af}(v)) -\ell(\mathfrak{af}(w))-\ell(\tilde w_\mu)\,.
$$
If $q(u,w,v)\neq 0$, then degree conditions on polynomial multiplication imply that
$c_{\tilde w_\mu,\tau^{-p}\mathfrak{af}(w)}^{\tilde w_\nu,k}=0$.
Otherwise, note that
$\tilde w_\nu(\tau^{-p}\mathfrak{af}(w))^{-1}=
\tau^{-p}R\mathfrak{af}(v)\mathfrak{af}(\id)^{-1}w^{-1}\tau^p=
\tau^{-p}(Rv)w^{-1}\tau^p$.  Since $\tau$ acts on reduced
expressions by adding a constant to each letter,
the number of affine factorizations of $\tau^{-p}(Rv)w^{-1}\tau^p$ 
and of $(Rv)w^{-1}$ are the same.
Moreover, if $(Rv)w^{-1}\in S_{\hat x}$ for some $x\in [n]$, 
then $\tilde w_\nu(\tau^{-p}\mathfrak{af}(w))^{-1}\in S_{\widehat{\tau^p x}}$.
The result thus follows from Theorem~\ref{theorem.main}.
\end{proof}

\begin{corollary}
\label{cor:GWall}
For $\d\in\mathbb N^{n-1}$ and $u,w,v\in S_n$ where $u$ has exactly one descent 
at position $r$, let $\mu' =\lambda(u)^\vee$.
If $\ell(v)\neq \ell(u)+\ell(w)-2\sum_id_i$
then $\langle u,w,w_0v\rangle_{\d}=0$.
Otherwise, if $\ell(u)\geq n(n-r-1)$, then
$$
\langle u,w,w_0v\rangle_{\d}=\# \;\text{of
highest weight factorizations of $(Rv)w^{-1}$ of weight 
$\mu$,}
$$
where $R= R_r \prod_{i=1}^{n-1} R_i^{d_{i-1}+d_{i+1}-2d_i}$.
\end{corollary}

\begin{proof}
We again use identification \eqref{gwonedescent}
of the invariants $\langle u,w,w_0v\rangle_{\d}$ with
$c_{\tilde w_{\mu},\tau^{-p}\mathfrak{af}(w)}^{\tilde w_\nu,k}$,
where $\tau^p\tilde w_\nu = R\mathfrak {af}(v)$.
The substitution of monomials in terms of Schur functions
(via the inverse Kostka matrix $\bar K$) into the monomial expansion
\eqref{dualkschur} of skew dual $k$-Schur functions
gives an alternating expression,
\begin{equation}
\label{klralt}
c_{\tilde w_\mu,\tilde w_\lambda}^{\tilde w_\nu,k} = 
\sum_{\alpha: |\alpha|=\ell(\tilde w_\nu \tilde w_\lambda^{-1})} 
\mathcal K_{\tilde w_\nu \tilde w_\lambda^{-1},\alpha}\, \overline K_{\alpha,\mu}
\,,
\end{equation}
for the affine LR coefficients.  Thus, they count a subset of 
affine factorizations of 
$\tilde v=\tilde w_\nu\mathfrak {af}(w)^{-1}\tau^p$ of weight $\alpha$, 
where $|\alpha|=|\mu|$.  Recall that an affine factorization of
$\tilde v$ has weight $\alpha$ only if
$\ell(\tilde v)=|\alpha|$.
Since $\ell(u)=|\lambda(u)|\geq n(n-r-1)$, 
we have that $|\mu|<n$ and hence $\ell(\tilde v)<n$.
But $\tilde v\in\tilde S_n$ and
therefore $\tilde v\in S_{\hat x}$
for some $x\in[n]$.
The result then follows by Theorem~\ref{the:GW}.
\end{proof}

Theorem~\ref{the:GW} and its corollary apply
to the problem of describing structure 
constants in the product of a Schur polynomial by a (quantum) Schubert polynomial
\cite{LS:1982,FGP:1997}.
Recall (e.g. \cite[Section 10.6, Proposition 8]{Fulton:1997})
that when $u\in S_n$ is Grassmannian with a descent at position $r$,
the Schubert polynomial $\mathfrak S_u$ is simply 
the Schur polynomial $s_{\lambda(u)}(x_1,\ldots,x_r)$.
Thus, we can address the coefficients in
$$
s_\lambda(x_1,\ldots,x_r)\,\mathfrak S_w = 
\sum_{v\in S_n} \langle u,w,w_0v\rangle_0\, \mathfrak S_{v}
\,,
$$
and in the quantum ($\d\neq 0$) analog.

In~\cite{MPP:2012}, the Fomin-Kirillov algebra is used to study 
the Gromov--Witten invariants $\langle u,w,w_0v\rangle_{\d}$ for
$u,v,w\in S_n$ where $u$ is Grassmannian and $\lambda(u)$ is a hook shape.
Their conditions on $u$ imply $|\lambda(u)|<n$ and thus suggest a
relation to a subset of the cases treated in Corollary~\ref{cor:GWall}.
However, despite satisfying a number of symmetry properties (e.g.~\cite{Post:2001}), 
the Gromov--Witten invariants are not symmetric under
the complementing of $\lambda(u)$ and there is no apparent relation.

Corollary~\ref{cor:GWall} can be used to give new results for the
classical case by setting $\d=0$.  For these, there is an even simpler 
highest weight formulation.

\begin{corollary}
\label{d=0}
Let  $u,v,w\in S_n$ where $u$ has exactly one descent at position $r$.
If $\ell(v)\neq \ell(u)+\ell(w)$, then $\langle u,w,w_0v\rangle_0=0$.
Otherwise, if either (i) $(R_r v)w^{-1}\in S_{\hat x}$ for some $x\in [n]$ or 
(ii) $\ell(\mu)=2$ for $\mu'=\lambda(u)^\vee$, then
$$
\langle u,w,w_0v\rangle_0=\# \;\text{of
highest weight factorizations 
of $(R_r v) w^{-1}$ of weight $\mu$.}
$$
\end{corollary}

\begin{example} 
Let $u=(1,2,4,7,3,5,6)$, $w=(3,1,5,4,2,6,7)$, and $v=(4,2,5,7,1,3,6)$ be 
permutations of $S_7$ in one-line notation.  Since $u$ is Grassmannian 
with its descent at position $r=4$, we find $\mu'=(3,3,2)$ by taking
the complement of $\lambda(u)=(3,1)$.  Note that
$\ell(w)=5$ and $\ell(v)=9$ and indeed $\ell(u)=\ell(v)-\ell(w)$.
By Corollary~\ref{d=0}, we compute the coefficient of $\mathfrak S_{v}$ 
in $s_{\lambda(u)}(x_1,\ldots,x_4)\mathfrak S_w$ by counting the highest weight 
factorizations of $\sigma=(R_4v)w^{-1}$ with weight $(3,3,2)$.

We note that $w^{-1}=(2,5,1,4,3,6,7)$ and use~\eqref{Ri} to find
$R_4v=[1,-1,2,4,5,7,10]$.  This gives $\sigma=s_6s_2s_3s_4s_3s_1s_2s_0$
whose affine factorizations $\sigma=v^3v^2v^1$ of weight $\mu=(3,3,2)$ 
satisfy $\ell(v^3)=2,\ell(v^2)=3,\ell(v^1)=3$ and each has a
word that is decreasing with respect to $4>3>2>1>0>6$ (since
$\sigma\in S_{\hat 5}$). Possible affine factorizations are:
$$
\content(v^3) \content(v^2) \content(v^1)
=(26)(431)(420)\quad (26)(310)(432) \quad (42)(316)(420)\quad (21)(436)(420)\;,
$$
and valid highest weights satisfy the extra condition that
all elements in factor $\content(v^i)$ are paired with an element of $\content(v^{i-1})$, for each $i>1$.  
Since $(26)(310)(432)$ is the only such factorization,
we find $\langle u,w,w_0v\rangle_0=1$.
\end{example}

\subsection{Quantum cohomology of the Grassmannian and fusion coefficients}

As with the quantum cohomology of full flags, the small quantum cohomology 
ring of the Grassmannian $\QH^*(\Gr(r,n))$ is a deformation of the 
usual cohomology.  As a linear space, this is the tensor product
$H^*(\Gr(r,n))\otimes\mathbb Z[q]$ and $\{\sigma_\lambda\}_{\lambda\subset (r^{n-r})}$
forms a $\mathbb Z[q]$-linear basis
of $\QH^*(\Gr(r,n))$.  Multiplication is a $q$-deformation of the product
in $H^*(\Gr(r,n))$, defined by
$$
\sigma_\lambda * \sigma_\mu = \sum_{\nu\subset (r^{n-r})\atop
|\nu|=|\lambda|+|\mu|-dn}
q^d\; \langle \lambda,\mu, \nu \rangle_d \,\sigma_{\nu^\vee}
\,.
$$
where the $\langle \lambda,\mu,\nu \rangle_d$ are the 3-point Gromov--Witten invariants
of genus 0 for the Grassmannian.  These constants
count the number of maps $f \colon \mathbb P_1 \to {\rm Gr}(r,n)$ 
whose image has degree $d$ and meets generic translates of Schubert varieties associated
to $\lambda,\mu$, and $\nu$.  

\begin{theorem}
\cite[Theorem 5.6]{LM:2008}
\label{gromovwitten}
For $\lambda,\mu,\nu \subset (r^{n-r})$,
$\langle \lambda,\mu,\nu \rangle_d = c_{w_\lambda,w_\mu}^{w_{\hat \nu},n-1}$,
where $\hat \nu$ is constructed from $\nu$ by adding
$d$ $n$-rim hooks, each starting in column $r$
and ending in the first column.  
\end{theorem}

This theorem combined with Theorem~\ref{theorem.main} shows that the crystal on affine factorizations applies directly 
the quantum cohomology of the Grassmannian. Furthermore, Postnikov~\cite{Posttoric:2005} defined cylindric Schur functions
$s_{\nu/d/\lambda}$ indexed by skew cylindric shapes $\nu/d/\lambda$ and proved that the 3-point genus 0 Gromov--Witten invariants 
for the Grassmannian $\langle \lambda,\mu,\nu \rangle_d$ appear in the Schur expansion 
of the toric Schur functions~\cite[Theorem 5.3]{Posttoric:2005} (which are restrictions to finitely many variables)
\begin{equation}
\label{equation.toric expansion}
       s_{\nu/d/\lambda}(x_1,\ldots,x_r) = \sum_{\mu \subseteq ((n-r)^r)} \langle \lambda, \mu, \nu \rangle_d \; s_\mu(x_1,\ldots,x_r)\;.
\end{equation}
Lam proved in \cite[Theorem 36]{Lam:2006} that the cylindric Schur
functions are precisely the subset of affine Stanley symmetric functions (or skew dual $k$-Schur functions)
whose index set is the affine permutations containing no braid relation $s_i s_{i+1} s_i$, also called
\textit{$321$-avoiding affine permutations}. This suggests that there is a crystal structure 
on 321-avoiding affine factorizations that would describe the 3-point, genus 0, Gromov--Witten invariants
of the Grassmannian in general without the conditions imposed by Theorem~\ref{theorem.main}.
Buch et al.~\cite{BKPT:2014} recently proved Knutson's puzzle rule for the Schubert structure constants on 
two-step flag varieties, which by~\cite[Corollary 1]{BKT:2003} yield Gromov--Witten invariants defining the (small) quantum 
cohomology ring $\QH^*({\rm Gr}(r,n))$ of a Grassmann variety in type $A$. 

It was proven~\cite{A:1995,BCF:1999} that the structure constants of the 
quantum cohomology of the Grassmannian are related to the fusion coefficients. 
We now reformulate our results in the fusion setting.
For $n>\ell\geq 1$, consider the quotient 
$R^{\ell n} = \Lambda_{(\ell)}/I^{\ell n}$,
where $I^{\ell n}$ is the ideal generated by
Schur functions that have exactly $n-\ell+1$
rows of length smaller than $\ell$:
$$
I^{\ell n}= \Bigl \langle s_{\lambda} \, 
\Big| \, 
 \# \{ j \, |  \, \lambda_j < \ell\}=n-\ell+1 \Bigr \rangle \, .
$$
The Verlinde (fusion) algebra of the WZW model associated 
to $\widehat{su}(\ell)$ at level $n-\ell$ is isomorphic to the quotient of
$R^{\ell n}$ modulo the single relation $s_{\ell}\equiv 1$ 
\cite{GW:1990,Kac,Walton:1990}. The fusion coefficient 
$\mathcal N_{\lambda, \mu}^{\nu}$ is defined for
$\lambda,\mu,\nu\subseteq ((n-\ell)^{\ell-1})$ as in~\eqref{equation.fusion}.
It was shown in~\cite{LM:2008} that
$$
\mathcal N_{\lambda, \mu}^{\nu}= c_{\tilde w_{\lambda'} \tilde w_{\mu'}}^{\tilde w_{\hat\nu},k}\,,
$$ 
where $\hat\nu=(\ell^{(|\lambda|+|\mu|-|\nu|)/\ell},\nu')$.

\begin{corollary}\cite{MS:2012}
Let $\lambda,\mu,\nu\subseteq ((n-\ell)^{\ell-1})$ and $\hat\nu=(\ell^{(|\lambda|+|\mu|-|\nu|)/\ell},\nu')$.
Whenever $\tilde w_{\hat \nu} \tilde w_{\lambda'}^{-1} \in S_{\hat x}$ for some $x\in [n]$ or $\ell(\mu')=2$,
the fusion coefficient $\mathcal N_{\lambda, \mu}^\nu$
is the number of highest weight factorizations of $\tilde w_{\hat \nu} \tilde w_{\lambda'}^{-1}$ of weight $\mu'$.
\end{corollary}

\subsection{Positroid stratification}
\label{positroid}

The real Grassmannian variety $\Gr(r,n)$ can be represented 
by the quotient space of $r\times n$ matrices (with at least one non-zero maximal minor) 
modulo the left action of the group ${\rm GL}_r$ of real $r\times r$ matrices with non-zero 
determinant.
The best understood cellular decomposition of the Grassmannian is given 
by the disjoint union of Schubert cells $\Omega_\lambda$ indexed by partitions 
$\lambda\subset ((n-r)^{r})$.  The totally nonnegative
Grassmannian $\Gr(r,n)_{\geq 0}$ is a subdivision of the Grassmannian $\Gr(r,n)$, 
presented in \cite{Post:2005} by the $r\times n$ matrices $X$ with all nonnegative 
Pl\"ucker coordinates $\Delta_I(X)$ (maximal $r\times r$ minors).  

For each point $X\in \Gr(r,n)_{\geq 0}$, the {\it positroid} of $X$ is defined by
$$
\mathcal M_X = \left\{I\in\binom{[n^+]}{r} \mid
\Delta_I(X)>0\right\}
\,,
$$
where $\binom{[n^+]}{r}$ is the set of $r$-element subsets of $\{1,\ldots,n\}$.
When the collection of points
$$
\mathring{\Pi}_{\mathcal M} = \left\{X\in \Gr(r,n)_{\geq 0} \mid
\mathcal M_X = \mathcal M\right\}
$$
is non-empty, it is called a {\it positroid cell}.  Postnikov established 
that the indexing sets for positroid cells are given by
various combinatorial objects such as Grassmann necklaces, 
Le diagrams, and plabic graphs.

It turns out that the complexification of positroid cells relates to the
projected Richardson stratification for the full flag manifold.  In particular, 
the positroid cells are {\it open positroid varieties}, 
defined in \cite{KLS:2013} to be an intersection of $n$ Schubert cells, 
taken with respect to the $n$ cyclic rotations of the standard flag.
The closure of the open positroid varieties are called the {\it positroid varieties}.
Knutson, Lam, and Speyer proved that the positroid varieties are
Richardson varieties projected to $\Gr(r,n)$ and they introduced yet 
another equinumerous indexing set;
the set of {\it bounded affine permutations},
$$
\text{Bound}(r,n) = \{ \tilde{w}\in \tilde S_{n,r} \mid i\leq \tilde{w}(i)\leq i+n\}\,.
$$
We shall denote positroid varieties by $\Pi_{\tilde{w}}$, for $\tilde{w}\in \text{Bound}(r,n)$,
and refer the reader to \cite{KLS:2013} for details on the map from
$\tilde{w}$ to $\mathcal M$.

Remarkably, the cohomology classes of positroid varieties can be 
represented by a projection of affine Stanley symmetric functions.  
Consider the map $\psi: \Lambda \rightarrow H^*(\Gr(r,n))$ 
where
$$
\psi(s_\lambda)= 
\begin{cases}
[\Omega_{\lambda}] & \;\text{ when } \lambda\subset (r^{n-r})\\
0 & \;\text{ otherwise}.
\end{cases}
$$
It was proven (\cite{KLS:2013}, Theorem 7.1) that
for each $\tilde w\in \operatorname{Bound}(r,n)$, 
$$\psi(F_{\tilde{w}})=[\Pi_{\tilde{w}}]\in H^*(\Gr(r,n))\,.
$$
It was also shown there that as long as  $\tilde w$ is
321-avoiding, the decomposition of positroid classes
into Schubert classes is given by 
Gromov--Witten invariants for the Grassmannian.
More generally, we use Proposition~\ref{af2klr} to show
that for any bounded $\tilde w$, the decomposition 
is given by certain $k$-Schur structure constants 
where now Gromov--Witten invariants for flags appear.

\begin{theorem}
Let $\tilde w\in \operatorname{Bound}(r,n)$.
The cohomology class of a positroid  decomposes over Schubert classes as
$$
[\Pi_{\tilde w}] = \sum_{\lambda\subset ((n-r)^r)}
a_{\tilde w,\lambda} \, [\Omega_\lambda] \,,
$$
where the set of $a_{\tilde w,\lambda}$ are affine LR-coefficients,
which include 3-point Gromov--Witten invariants for the complete flag $\Fl_n$.
\end{theorem}

\begin{proof}
Recall that the cohomology $H^*(\Gr(r,n))$ can be presented as the quotient
$$
\phi \colon H^*(\Gr(r,n)) \cong \Lambda/\mathcal I\,,
$$
where the ideal $\mathcal I = \langle e_i,h_j\mid i>r,j>n-r \rangle$.
Since $\mathcal I$ is spanned by the Schur functions $s_\lambda$
whose shapes do not fit inside the $r\times (n-r)$-rectangle, isomorphism 
$\phi$ is given by identifying the Schubert class $[\Omega_\lambda]$ 
to (the coset of) the Schur function $s_\lambda$.
The map $\psi$ thus decomposes into the composition of $\phi$ and $\bar\psi$,
where $\bar \psi$ is the quotient map from $\Lambda$ to $\Lambda/\mathcal I$ defined by 
annihilating each $s_\lambda$ for $\lambda\not\subset ((n-r)^r)$.
Thus we have 
$$
\bar\psi(F_{\tilde w}) = \sum_{\mu\subset ((n-r)^r)} a_{\tilde w,\mu} \,s_\mu
$$
and by Proposition~\ref{af2klr} the coefficients $a_{\tilde w,\mu}$ are affine LR-coefficients.
\end{proof}

Applying Theorem~\ref{the:GW} (or Corollary~\ref{corollary.a}) to this result, we can
explicitly describe the decomposition in terms of
the crystal in many cases.

\begin{corollary}
When $\tilde w\in S_{\hat x}\subseteq \operatorname{Bound}(r,n)$ for 
some $x\in [n]$, 
$$
[\Pi_{\tilde w}] = \sum_{\lambda\subset((n-r)^r)} a_{\tilde w,\lambda} \,[\Omega_\lambda]\,,
$$
where $a_{\tilde w,\lambda}=\left|\{v\in \mathcal W_{\tilde w,\lambda} \mid \et_i(v)={\bf 0}\;\forall i\}\right|$
counts highest weight factorizations of the crystal $B(\tilde w)$.
\end{corollary}

\section{Crystal operator involution}
\label{section.involution}

The broad goal of this article is to introduce the powerful
theory of crystals into the combinatorial study of affine
Schubert calculus and Gromov--Witten invariants.  However,
Proposition~\ref{af2klr} suggests that existing crystal 
theory is not enough to address invariants 
$\langle u,w,v\rangle_d$ without the limitation that
$u$ is Grassmannian.  It is with this in mind that 
we give here an alternate proof of Theorem~\ref{theorem.main} that
circumvents the need for Theorem~\ref{theorem.highestwts}.

An elegant proof of the formulation of the Littlewood--Richardson 
rule as highest weights in the $\mathfrak {sl}_{\ell}$-crystal
of Section~\ref{section.specht} was given by Remmel 
and Shimozono~\cite{RS:1998}.  They substitute a Schur function,
in the product of Schur functions, by its expansion 
in terms of
homogeneous symmetric functions to obtain
$$
	s_\lambda s_\mu = 
	s_\lambda \sum_\alpha \overline K_{\alpha, \mu} \;h_\alpha\
	= \sum_{\nu,\alpha} K_{\nu/\lambda,\alpha}\overline K_{\alpha,\mu} \; s_\nu\,.
$$
Here, the Kostka numbers $K_{\nu/\lambda,\alpha}$ enumerate the set $\SSYT(\nu/\lambda,\alpha)$
of skew tableaux of shape $\nu/\lambda$ and weight $\alpha$,
whereas $\overline{K}_{\alpha,\mu}$ is an alternating sum.
To be precise, let $m$ be a positive integer which weakly exceeds the 
length of the partitions $\alpha$ and $\mu$, let $S_m$ be the symmetric 
group of permutations in $m$ letters, and define $\rho=(m-1,m-2,\ldots,1,0)$.
Then in~\cite[Eq. (1.11)]{RS:1998} the inverse Kostka matrix was expressed as
\begin{equation}
\label{equation.inverse kostka}
	\overline{K}_{\alpha,\mu}= 
	\sum_{\sigma\in S_m}
(-1)^{{\rm sign}(\sigma)} \;,
\end{equation}
over permutations $\sigma$ where $\sigma(\mu+\rho)-\rho \in S_m \alpha$.
From this,
$$
s_\lambda s_\mu = 
\sum_{\nu, \alpha} 
\sum_{(T,\sigma)}
(-1)^{{\rm sign}(\sigma)}  s_\nu\,,
$$
where $\sigma\in S_m$ such that $\sigma(\mu+\rho)-\rho \in S_m \alpha$
and $T\in \SSYT(\nu/\lambda,\alpha)$.
The trick to canceling the negative terms lies in an involution that is defined using $\st_i \et_i$ for a suitable $i$,
where $\et_i$ are the crystal raising operators on tableaux and $\st_i(T) = \et_i^p(T)$ if $p:=\ve_i(T)-\vp_i(T)\ge 0$
and $\st_i(T) = \ft_i^{-p}(T)$ if $p<0$ are the reflections within an $i$-string.
The Littlewood--Richardson rule follows since the Yamanouchi tableaux are fixed points under the involution.

In the same spirit, we produce a sign-reversing involution 
using the crystal operators on affine factorizations.

\begin{definition}
\label{definition.theta}
For any $\mu\subset(a^{n-a})$ and 
$w\in S_{\hat x}$ for some fixed $x\in[n]$,
define $\theta$ to act on
the set of pairs
$\bigcup_\beta\{\sigma \in S_m \mid \sigma(\mu+\rho)-\rho=\beta\}
\times \mathcal W_{w,\beta}$ for $m>\ell(w)$
by
\begin{equation*}
\theta (\sigma,w^\beta)=
\begin{cases}
(\sigma,w^\beta)&\text{when}\quad 
X=\emptyset,
\\
(s_r\sigma,\st_r \et_r(w^\beta))& \text{otherwise}
\,,
\end{cases}
\end{equation*}
where $X= \bigcup_i L_i(w^\beta)$ and
$r=\max\{i \mid \max(X)\in L_i(w^\beta)\}$.
\end{definition}

\begin{prop}
\label{signrev}
For any $\mu\subset(a^{n-a})$ and 
$w\in S_{\hat x}$ for some fixed $x\in[n]$,
$\theta$ is an involution 
whose fixed points are 
$\{(\id,w^\mu) \mid w^\mu \text{ is a highest weight
factorization of $w$ with weight $\mu$}\}\,.$
\end{prop}

\begin{proof}
Consider a factorization $w^\beta$ of $w\in S_{\hat x}$
with $r=\max\{i \mid \max(X)\in L_i(w^\beta)\}$ and
let $w^\gamma=\tilde s_r\tilde e_r(w^\beta)$.
It is known for crystal operators (see for example~\cite{RS:1998}) that
$(\st_r \et_r)^2$ either acts as the identity or annihilates an element.
Hence our main task in verifying that $\theta$ is an involution
is to prove
$$
\theta^2(\sigma,w^\beta)=(s_r^2\sigma,(\st_r\et_r)^2(w^\beta))
$$
when $X\neq \emptyset$
since $(s_r^2\sigma,(\st_r\et_r)^2(w^\beta)) =(\sigma,w^\beta)$
by the above argument.
To do so, we must show that
$r=\max\{i \mid \max(Y)\in L_i(w^\gamma)\}\;
\text{ for $Y=\bigcup_i L_i(w^\gamma)$}$ when $X\neq \emptyset$.

Let $u v= w^{\beta_{r+1}}w^{\beta_r}$ and $UV=\st_1\et_1(uv)$
and set $z=\max(L_1(uv))$.  Consider 
\begin{equation}
u=(u_1 s_z u_2) \quad \text{and}\quad v=(v_1v_2)
\end{equation}
for $\content(u_1)=\{c\in\content(u) \mid c>z\}$
and
$\content(v_1)=\{c\in\content(v) \mid c\geq z\}.$
We start by showing that $\max(L_1(UV))=z$ and
that
\begin{equation}
\label{UV}
U=(u_1 s_z U_2) \quad \text{and}\quad V=(v_1 V_2)
\,,
\end{equation}
where the elements in $\content(U_2)$ and $\content(V_2)$
are smaller than $z$.
To this end, let $\tilde u \tilde v = \et_1(uv)$ 
and set $p=|L_1(\tilde u\tilde v)|>0$ and $q=|R_1(\tilde u\tilde v)|$.
Then
$$
U V=
\begin{cases}
\et_1^{p-q}(\tilde u\tilde v) & \text{if $p>q$,}\\
\tilde u\tilde v & \text{if $p=q$,}\\
\ft_1^{q-p}(\tilde u\tilde v) & \text{if $p<q$}
\,.
\end{cases}
$$
Let $b=\min(L_1(uv))$ and
$t=\max\{i\geq 0 \mid b-i-1\not\in\content(u)\}$.
\eqref{tuv} tells us that $\tilde u$ differs from $u$ by
the deletion of $b$ and $\tilde v$ differs from $v$
by the addition of $b-t$.
We deduce from Lemma~\ref{lem:uvdecomp} that 
$\max(R_1(\tilde u\tilde v)))=b-t$ (implying that $q>0$)
and $\max(L_1(\tilde u\tilde v))=z$ unless $L_1(\tilde u\tilde v)=\emptyset$.

When $p=q$, $q>0$ implies that
$\max(L_1(UV))=\max(L_1(\tilde u\tilde v))=z$
and \eqref{UV} holds.
When $p>q$, we can iterate our deduction
to find that $\max(L_1(UV))=z$ (since $p-q<p$) and that
$U$ differs from $u$ by the deletion of letters smaller
than $z$ and $V$ differs from $v$ by the addition 
of letters smaller than $z$ implying \eqref{UV}.
If $p<q$, we have that $UV= \ft_1^{q-p-1}(uv)$ since
$\ft_1\et_1=\id$ by Proposition~\ref{prop:stayput}.
Lemma~\ref{lem:Ruvdecomp} can then be used to prove
our claim in this last case.

Now given that $\max(L_1(UV))=z$, 
since $\st_r$ and $\et_r$ act on $w^\beta$ by changing 
only factors $uv$ into $UV$, it suffices to verify
$$
\max\{i \mid \max(Y)\in L_i(\tau UV\sigma)\}=2\;\text{
for $Y=\bigcup_i L_i(\tau UV \sigma)$}\,,
$$
where $\tau=w^{\beta_{r+2}}$ and $\sigma=w^{\beta_{r-1}}$.
Consider the decomposition
$$
\tau \,u = (\tau_1  \tau_2) (u_1 s_z u_2)
\,,
$$
where $\content(\tau_1)=\{c\in\content(\tau) \mid c \ge z\}$.
If $\tilde z=\max(L_1(\tau u))$, then our
choice of $r$ implies $\tilde z< z$.  Therefore, every 
element in $\tau_1$ is paired with something in $u_1$ in the
$\tau\,u$ pairing.  Since 
$\tau \, U = (\tau_1  \tau_2) (u_1 s_z U_2)$,
we find that $\max(L_1(\tau U))<z$.

Next consider
$$
v \sigma = (v_1  v_2) (\sigma_1 \sigma_2)
\,,
$$
where $\content(v_1)=\{c\in\content(v) \mid c\geq z\}$
and $\content(\sigma_1)=\{c\in\content(\sigma) \mid c > z\}$.
Since our choice of $r$ implies here that
$\max(L_1(v\sigma)) \ge z$, we have that
every element in $\content(v_1)$ is paired
with something in $\content(\sigma_1)$.
Since $V=(v_1 V_2)$, we have that $\max(L_1(V\sigma)) \le z$.

Let $(\sigma,w^\beta)$ be a fixed point of $\theta$. 
If $\bigcup_i L_i(w^\beta)=\emptyset$,
then every element of $\content(w^{\beta_i})$ is paired 
with something in $\content(w^{\beta_{i-1}})$ and
in particular, $\beta_{i-1}\leq\beta_{i}$ for all $i$.
Given $\sigma(\mu+\rho)-\rho=\beta$, we have that
$\beta_j=\mu_i+j-i$ where $\sigma(i)=j$.  Since $\beta$ 
and $\mu$ are partitions, $\beta_1=\max\{\mu_i+1-i\}=\mu_1$ 
implying that $\sigma(1)=1$.  By iteration, $\sigma=\id$ and
$\beta=\mu$.  By Definition~\ref{def:yama},
$w^\beta$  is a highest weight factorization of weight $\mu$.
\end{proof}

\begin{remark}
Note that the proof of Proposition~\ref{signrev} goes through in almost the identical
manner if we defined $r=\min\{i \mid \max(X)\in L_i(w^\beta)\}$ instead of
$r=\max\{i \mid \max(X)\in L_i(w^\beta)\}$.
In the $k\to \infty$ limit, these choices correspond to choosing the rightmost
violation of the Yamanouchi word condition in the row and column reading of the 
tableau, respectively. In~\cite{RS:1998}, row reading was chosen, so
$r=\min\{i \mid \max(X)\in L_i(w^\beta)\}$  in our formulation.
\end{remark}

The previous proposition immediately implies that
$$
\sum_\beta 
\sum_{(\sigma,w^\beta)\atop
\theta(\sigma,w^\beta)\neq(\sigma,w^\beta)}
(-1)^{{\rm sign}(\sigma)} 
=0
\,,
$$
leaving only the $\theta$-fixed points $(\id,w^\mu)$, where
$w^\mu$ is a highest weight factorization.

\begin{corollary}
\label{corollary.cancel}
For any $\mu\subset(a^{n-a})$ and 
$w\in S_{\hat x}$ for some fixed $x\in[n]$,
$$
\sum_{\beta}
\sum_{(\sigma,w^\beta)}
(-1)^{{\rm sign}(\sigma)} 
=
\#\{ w^\mu \in \mathcal W_{w,\mu} \mid \text{$w^\mu$ is highest weight } \}
\,,
$$
where the sum is over pairs $(\sigma,w^\beta)$ with 
$\sigma(\mu+\rho)-\rho=\beta$ and $w^\beta$ is an 
affine factorization of $w$ with weight $\beta$.
\end{corollary}

The involution $\theta$ allows us to give an alternate proof of 
Theorem~\ref{theorem.main}.

\begin{proof}[Alternative proof of Theorem~\ref{theorem.main}]
Since \eqref{kschurrec} implies that $c_{Ru,w}^{Rv,k}=c_{u,w}^{v,k}$, we
are only concerned with products $s_\mu s_w^{(k)}$ for
some $\mu\subset (a^{n-a})$.
Substitution of formula \eqref{equation.inverse kostka}
for inverse Kostka numbers into the $h$-expansion of a Schur function
$s_\mu = \sum_\alpha \overline{K}_{\alpha, \mu} \;h_\alpha$ gives 
$$
s_{\mu} = 
\sum_{\alpha \in\mathcal P^n}
\sum_{\sigma\in S_m\atop
\sigma(\mu+\rho)-\rho\in S_m\alpha}
(-1)^{{\rm sign}(\sigma)} 
\, h_{\alpha}\,,
$$
where $\rho=(m-1,m-2,\ldots,1,0)$ and $m\le n$ is weakly 
bigger than the number of parts in $\mu$.  This given,
we use \eqref{iteratekpieri} to $k$-Schur expand
the product of $h_\alpha$ with a $k$-Schur function $s_w^{(k)}$
in 
$$
s_{\mu}\,s_w^{(k)} = 
\sum_{\alpha \in\mathcal P^n}
\sum_{\sigma\in S_m\atop
\sigma(\mu+\rho)-\rho\in S_m\alpha}
(-1)^{{\rm sign}(\sigma)} \, h_{\alpha}\,s_w^{(k)}=
\sum_{v\in\tS}
\sum_{\alpha\in\mathcal P^n}
\mathcal K_{vw^{-1},\alpha}
\!  \!
\sum_{\sigma\in S_m\atop\sigma(\mu+\rho)-\rho\in S_m\alpha}
(-1)^{{\rm sign}(\sigma)} \, 
s^{(k)}_v
\,.
$$
In fact, $\mathcal K_{vw^{-1},\alpha}=\mathcal K_{vw^{-1},\beta}$ for 
any rearrangement $\beta$ of $\alpha$ leading us 
to the expression 
$$
s_{\mu}\,s_w^{(k)} = 
\sum_{v\in\tS}
\sum_{\beta}
\sum_{(\sigma,w^\beta)} 
(-1)^{{\rm sign}(\sigma)} \, s^{(k)}_{v}\,,
$$
where the sum is over pairs $(\sigma,w^\beta)$ with 
$\sigma(\mu+\rho)-\rho=\beta$ and $w^\beta$ is an 
affine factorization of $vw^{-1}$ with weight $\beta$.
Corollary~\ref{corollary.cancel} yields the desired result.

Note that if $\mu$ has only two parts, we can choose $m=2$ and in this
case $\theta$ is still defined as in Definition~\ref{definition.theta} with $r=1$ by 
Section~\ref{subsection.two factor}.
\end{proof}

\appendix
\section{Appendix}
\label{appendix.stembridge}

\begin{proof}[Proof of Theorem~\ref{theorem.crystal}]
The proof proceeds by checking Stembridge's local axioms of Section~\ref{subsection.stembridge}.
We freely use the properties of the operators $\et_i$ and $\ft_i$ established in Section~\ref{subsection.properties}.
The lengths of the monochromatic directed paths are given by $\ve_r(w^\beta)$ and $\vp_r(w^\beta)$
for every $w^\beta \in \mathcal W_w$. By Proposition~\ref{prop:stayput} (3) 
these are given by the number
of unbracketed letters, which is finite. This shows (P1) of Stembridge's local axioms.
(P2) is also ensured by the definition of the crystal operators on affine factorizations.

\subsubsection*{Proof of (P3) and (P4)}
Next we consider axioms (P3) and (P4) by proving that
\begin{equation*}
	(a_{ij},\Delta_i\ve_j(w^\beta),\Delta_i\vp_j(w^\beta))\in\{(0,0,0),(-1,-1,0),(-1,0,-1)\}\,.
\end{equation*} 
If $a_{ij}=0$, then $\et_i$ and $\et_j$ (resp. $\ft_j$) commute, so that indeed
$(a_{ij},\Delta_i\ve_j(w^\beta),\Delta_i\vp_j(w^\beta))=(0,0,0)$ in this case. Next consider $a_{ij}=-1$, so
that $j=i-1$ or $j=i+1$. First assume $j=i-1$. Without loss of generality we may assume that 
$w^\beta = yuv$ is a product of three factors and $i=2$. Since by assumption $\et_2(yuv)$ is defined,
one letter moves from factor $y$ to factor $u$. Call $s_c$ the new generator in $u$ under $\et_2$.
Recall that as in Lemma~\ref{lem:uvdecomp}, we may write
\begin{equation}
\label{equation.uv stembridge}
	uv = (u_1 s_b\cdots s_{b-t} u_2)(v_1s_b \cdots s_{b-t+1} v_2),
\end{equation}
where $b=\min(L_1(uv))$, all letters in $\content(v_1)$ are paired with something in $\content(u_1)$,
and every element in $\content(u_2)$ is with with something in $\content(v_2)$. Note that 
\begin{enumerate}
\item If $c>b$ we have $\ve_1(\et_2 yuv)=\ve_1(yuv)+1$ and $\vp_1(\et_2 yuv)=\vp_1(yuv)$ since still
all letters in $\content(v_1)$ are paired and there is one extra unpaired letter in $\content(u_1)\cup\{c\}$ 
after the application of $\et_1$. Hence $(a_{21},\Delta_2\ve_1(yuv),\Delta_2\vp_1(yuv))=(-1,-1,0)$.
\item If $c<b-t$, we have two cases:
\begin{enumerate}
\item If $c$ does not pair with a letter in $\content(v_2)$, then as before $\ve_1(\et_2 yuv)=\ve_1(yuv)+1$ and 
$\vp_1(\et_2 yuv)=\vp_1(yuv)$, so that again $(a_{21},\Delta_2\ve_1(yuv),\Delta_2\vp_1(yuv))=(-1,-1,0)$.
\item If $c$ does pair with a letter in $\content(v_2)$, then $\ve_1(\et_2 yuv)=\ve_1(yuv)$ and 
$\vp_1(\et_2 yuv)=\vp_1(yuv)-1$, so  $(a_{21},\Delta_2\ve_1(yuv),\Delta_2\vp_1(yuv))=(-1,0,-1)$.
\end{enumerate}
\end{enumerate}
This proves (P3) and (P4) for $j=i-1$.

Now assume that $j=i+1$. Without loss of generality we may assume that $w^\beta=uvy$ is a product of
three factors and $i=1$. Since by assumption $\et_1(uvy)$ is defined, one letter moves from factor $v$ to
factor $y$ under $\et_1$. Call $s_c$ the generator in $v$ that disappears. We can write $uv$ again as
in~\eqref{equation.uv stembridge}. Note that
\begin{enumerate}
\item If $c\in \content(v_1)$, then $\ve_2(\et_1 uvy)=\ve_2(uvy)+1$ and $\vp_2(\et_1 uvy)=\vp_2(uvy)$ since
one less letter is bracketed in $\content(u_1)$. Hence $(a_{12},\Delta_1\ve_2(uvy),\Delta_1\vp_2(uvy))=(-1,-1,0)$.
\item If $c\in \{b-t+1,b-t+2,\ldots,b\}$, then the letter $c-1$ becomes unbracketed in $\content(u)$, so that 
$\ve_2(\et_1 uvy)=\ve_2(uvy)+1$ and $\vp_2(\et_1 uvy)=\vp_2(uvy)$ and hence 
$(a_{12},\Delta_1\ve_2(uvy),\Delta_1\vp_2(uvy))=(-1,-1,0)$.
\item If $c\in \content(v_2)$, we have two cases
\begin{enumerate}
\item If $c$ is paired with a letter $c'\in\content(u_2)$ and $c'$ does not find a new bracketing 
partner in $\content(v_2)$ after $c$ is removed, then $\ve_2(\et_1 uvy)=\ve_2(uvy)+1$ and 
$\vp_2(\et_1 uvy)=\vp_2(uvy)$ and hence $(a_{12},\Delta_1\ve_2(uvy),\Delta_1\vp_2(uvy))=(-1,-1,0)$.
\item If $c$ is paired with a letter $c'\in\content(u_2)$, but $c'$ finds a new bracketing partner in 
$\content(v_2)\setminus \{c\}$, or if $c$ is not bracketed with a letter in $\content(u_2)$, then
$\ve_2(\et_1 uvy)=\ve_2(uvy)$ and $\vp_2(\et_1 uvy)=\vp_2(uvy)-1$, so that 
$(a_{12},\Delta_1\ve_2(uvy),\Delta_1\vp_2(uvy))=(-1,0,-1)$.
\end{enumerate}
\end{enumerate}
This completes the proof of (P3) and (P4) for $j=i+1$.

\subsubsection*{Proof of (P5)}
Next we prove (P5). When $a_{ij}=0$, then $\et_i$ commutes with $\et_j$ and $\ft_j$, so that the conditions
of (P5) follow easily. Next assume that $j=i-1$. It suffices again to consider $w^\beta=yuv$ with $i=2$. 
Then by the analysis in (P3) and (P4) above, we have that $\Delta_2\ve_1(yuv)=0$ only when $c<b-t$
and $c$ pairs with a letter in $\content(v_2)$. In this case $\et_1$ moves the letter $b$ from $u$ to the letter
$b-t$ in $v$ before and after $\et_2$, so that $\et_1\et_2(yuv)=\et_2\et_1(yuv)$.
In addition $\nabla_1 \vp_2(\et_1\et_2 yuv)= \vp_2(\et_2\et_1yuv) - \vp_2(\ft_1\et_1\et_2 yuv)
=\vp_2(\et_1yuv)+1-\vp_2(yuv)-1=0$ since by Lemma~\ref{lem:uvdecomp} all letters in
$\content(u_1 s_b\cdots s_{b-t})$ are bracketed with letters in $y$ and hence do contribute
to neither $\vp_2(yuv)$ nor $\vp_2(\et_1yuv)$ (and $\et_1$ moves $b$). This proves (P5) when $j=i-1$.

Now assume that $j=i+1$. It suffices to consider $w^\beta=uvy$ with $i=1$, and by the above analysis
of cases (P3) and (P4) we have $\Delta_1\ve_2(uvy)=0$ only when $c\in \content(v_2)$ and either $c$
is paired with a letter $c'\in\content(u_2)$, but $c'$ finds a new bracketing partner in 
$\content(v_2)\setminus \{c\}$ after the application of $\et_1$, or if $c$ is not bracketed with a letter in $\content(u_2)$.
Since the letters in $\content(v_2)$ remain bracketed and $b$ from $\content(u)$ is moved to $b-t$ in $\content(v)$,
we have $\et_1\et_2(uvy)=\et_2\et_1(uvy)$. A similar computation as in the case $j=i-1$ shows that
$\nabla_2 \vp_1(\et_1\et_2 uvy)= 0$. This concludes the proof of (P5).

\subsubsection*{Proof of (P6)}

For the proof of (P6) assume that $\Delta_i \ve_j(w^\beta) = \Delta_j \ve_i(w^\beta) = -1$. Without loss of generality we
may assume that $w^\beta=yuv$, $i=2$, and $j=1$. Then let us write 
\begin{equation*}
\begin{split}
 yuv & = y (u_1 s_{b_1} \cdots s_{b_1-t_1} u_2)(v_1 s_{b_1} \cdots s_{b_1-t_1+1} v_2)\\
 yuv & = (y_1 s_{b_2} \cdots s_{b_2-t_2} y_2)(u'_1 s_{b_2} \cdots s_{b_2-t_2+1} u'_2) v
\end{split}
\end{equation*}
in the decomposition according to Lemma~\ref{lem:uvdecomp} with respect to $\et_1$ and $\et_2$, respectively.
All letters in $\content(v_1)$ pair with letters in $\content(u_1)$ and all letters in $\content(u_2)$ pair with letters in $\content(v_2)$.
In addition, all letters in $\content(u_1')$ pair with letters in $\content(y_1)$ and all letters in $\content(y_2)$ pair with letters in $\content(u_2')$.

By the analysis of (P3) and (P4), $\Delta_2 \ve_1(yuv)=-1$ unless the new letter $b_2-t_2$ in $\content(u)$ under $\et_2$ satisfies $b_2-t_2<b_1-t_1$
and pairs with a letter in $\content(v_2)$. This implies in particular that $\content(u_2)\subseteq \content(u_2')$.

By the analysis of (P3) and (P4), $\Delta_1 \ve_2(yuv)=-1$ if $b_1 \not \in \content(u_2')$ or $b_1\in \content(u_2')$ but $b_1$ is paired with $c'\in y_2$
for the $\et_2$ bracketing and $c'$ cannot find a new bracketing partner in $u$ when $b_1$ is removed by $\et_1$. 

First assume that $b_2-t_2 \ge b_1 - t_1$. Since $b_2-t_2$ is the new letter in $u$ under $\et_2$ this means in particular that
$b_2-t_2>b_1$ (since all letters $b_1-t_1,\ldots, b_1-1,b_1$ already appear in $\content(u)$). Hence $b_1 \in \content(u_2')$ and we can
write $yuv$ as 
\[
	(y_1 s_{b_2} \cdots s_{b_2-t_2} \overbrace{y_2}^{s_{c'}\in}) (u_1' s_{b_2} \cdots s_{b_2-t_2+1} \overbrace{\tilde{u}_2 s_{b_1} \cdots s_{b_1-t_1} u_2}^{u_2'}) 
	(v_1 s_{b_1} \cdots s_{b_1-t_1+1} v_2),
\]
where possibly $\tilde{u}_2=1$. The letter $b_1$ is paired with $c'\in \content(y_2)$. Since $c'$ cannot find a new bracketing partner when 
$b_1$ is removed by $\et_1$, all letters in $\content(\tilde{u}_2)$ must be paired with letters in $\content(y_2)$. Now computing 
$\et_1\et_2^2 \et_1(yuv)$ we obtain
\begin{itemize}
\item under $\et_1$ the letter $b_1$ moves from $u$ to $b_1-t_1$ in $v$;
\item under $\et_2$ the letter $c'$ moves from $y$ to some letter $c''$ in $u$; since it is smaller than $b_1-t_1$ it must bracket with $b_1-t_1$ in $v$
(or some other letter in $v$);
\item under $\et_2$ the letter $b_2$ moves from $y$ to $b_2-t_2$ in $u$;
\item under $\et_1$ the rightmost unbracketed letter $b_2-t_2 \ge i>b_1$ moves from $u$ to $v$.
\end{itemize}
Next computing $\et_2 \et_1^2 \et_2(yuv)$ we obtain
\begin{itemize}
\item under $\et_2$ the letter $b_2$ moves from $y$ to $b_2-t_2$ in $u$;
\item under $\et_1$ the letter $b_1$ moves from $u$ to $b_1-t_1$ in $v$;
\item under $\et_1$ the rightmost unbracketed letter $b_2-t_2 \ge i>b_1$ moves from $u$ to $v$;
\item under $\et_2$ the letter $c'$ moves from $y$ to $c''$ in $u$; again it is bracketed with a letter in $v$.
\end{itemize}
This shows that $z:=\et_1\et_2^2 \et_1(yuv)=\et_2 \et_1^2 \et_2(yuv)$. It remains to verify that $\nabla_1 \vp_2(z)=\nabla_2 \vp_1(z)=-1$.
By the above explicit description of the action of $\et_1$, we find that $\vp_2(\et_2^2 \et_1 yuv) = \vp_2(\et_1 yuv)+2 = \vp_2(yuv)+2$.
When we act with $\et_1$ on $\et_2^2 \et_1 yuv$, the rightmost unbracketed letter $b_2-t_2 \ge i>b_1$ moves from $u$ to $v$.
If $i=b_2-t_2$, then certainly $\vp_2(z)=\vp_2(yuv)+1$. Otherwise the letter in $\content(y)$ which was before bracketed with $i$, 
brackets with $b_2-t_2$ after the application of $\et_1$. Hence again $\vp_2(z)=\vp_2(yuv)+1$. Altogether $\nabla_1 \vp_2(z)=-1$.
Similarly, $\vp_1(\et_1^2 \et_2 yuv) = \vp_1(\et_2 yuv) +2 = \vp_1(yuv)+2$. When we act with $\et_2$ on $\et_1^2\et_2 yuv$, the new
letter $c''$ brackets with $b_1-t_1$ or another letter in $v_2$. Hence $\vp_1(\et_2 \et_1^2 \et_2 yuv) = \vp_1(yuv)+1$. This implies 
$\nabla_2 \vp_1(z)=-1$.
This concludes the proof of (P6) when $b_2-t_2 \ge b_1 - t_1$.

Next assume that $b_1-t_1=b_2-t_2+1$. In this case $s_{b_2} \cdots s_{b_2-t_2+1}$ and $s_{b_1} \cdots s_{b_1-t_1}$ in $u$
overlap. In particular $b_1\le b_2$, so that $yuv$ can be written as
\[
	(y_1 s_{b_2}\cdots s_{b_2-t_2} y_2) ( \overbrace{u_1' s_{b_2} \cdots }^{u_1}s_{b_1} \cdots \underbrace{s_{b_2-t_2+1}}_{=s_{b_1-t_1}} \overbrace{u_2}^{=u_2'})
	(v_1 s_{b_1} \cdots s_{b_1-t_1+1} v_2)\;.
\]
Computing $\et_1 \et_2^2 \et_1(yuv)$ we obtain
\begin{itemize}
\item under $\et_1$ the letter $b_1$ moves from $u$ to $b_1-t_1$ in $v$;
\item under $\et_2$ the letter $b_1-1$ moves from $y$ to $b_2-t_2$ in $u$;
\item under $\et_2$ the letter $b_2$ moves from $y$ to $b_1$ in $u$;
\item under $\et_1$ the letter $b_1$ moves from $u$ to a letter $i\le b_1-t_1-1$ in $v$.
\end{itemize}
Similarly, computing $\et_2 \et_1^2 \et_2(yuv)$ we have
\begin{itemize}
\item under $\et_2$ the letter $b_2$ moves from $y$ to $b_2-t_2$ in $u$;
\item under $\et_1$ the letter $b_2-t_2=b_1-t_1-1$ moves from $u$ to a letter $i\le b_1-t_1-1$ in $v$;
\item under $\et_1$ the letter $b_1$ moves from $u$ to a letter $b_1-t_1$ in $v$;
\item under $\et_2$ the letter $b_1-1$ moves from $y$ to $b_2-t_2$ in $u$.
\end{itemize}
This implies that $z:=\et_1\et_2^2 \et_1(yuv)=\et_2 \et_1^2 \et_2(yuv)$. By very similar arguments to the previous case we also have
$\nabla_1 \vp_2(z)=\nabla_2 \vp_1(z)=-1$. This concludes the proof of (P6) when $b_1-t_1=b_2-t_2+1$.

Finally assume that $b_2-t_2<b_1-t_1-1$. Then we have
\[
	(y_1 s_{b_2}\cdots s_{b_2-t_2} y_2) (\overbrace{u_1 s_{b_1}\cdots s_{b_1-t_1} \tilde{u}_1}^{=u_1'} s_{b_2} \cdots s_{b_2-t_2+1} u_2')
	(v_1 s_{b_1} \cdots s_{b_1-t_1+1} v_2).
\]
Computing $\et_1 \et_2^2 \et_1(yuv)$ we obtain
\begin{itemize}
\item under $\et_1$ the letter $b_1$ moves from $u$ to the letter $b_1-t_1$ in $v$;
\item under $\et_2$ the letter $b_2$ moves from $y$ to $b_2-t_2$ in u. Since by assumption $\Delta_2 \ve_1(yuv)=-1$, the new letter $b_2-t_2$ 
in $u$ does not pair with a letter in $\content(v_2)$;
\item since $b_1$ was moved from $u$ by $\et_1$, there is at least one free letter in $\content(y_1)$ which is not bracketed with a letter in $u$. 
Let $c'$ be the smallest such letter. Under $\et_2$ the letter $c'$ moves from $y$ to a letter $c''<b_1-t_1$ in $u$;
\item under $\et_1$ the letter $b_2-t_2$ moves from $u$ to a letter $i\le b_2-t_2$ in $v$.
\end{itemize}
Next computing $\et_2 \et_1^2 \et_2(yuv)$ yields
\begin{itemize}
\item under $\et_2$ the letter $b_2$ moves from $y$ to $b_2-t_2$ in $u$; again since $\Delta_2 \ve_1(yuv)=-1$, the new letter $b_2-t_2$ 
in $u$ does not pair with a letter in $\content(v_2)$;
\item under $\et_1$ the letter $b_2-t_2$ in $u$ moves to a letter $i\le b_2-t_2$ in $v$;
\item under $\et_1$ the letter $b_1$ in $u$ moves to $b_1-t_1$ in $v$;
\item under $\et_2$ the same letter $c'$ from the previous case moves from $y$ to a letter $c''<b_1-t_1$ in $u$.
\end{itemize}
Again, this show that $z:=\et_1\et_2^2 \et_1(yuv)=\et_2 \et_1^2 \et_2(yuv)$. By very similar arguments to the previous case we also have
$\nabla_1 \vp_2(z)=\nabla_2 \vp_1(z)=-1$. This concludes the proof of (P6).

\subsubsection*{Proof of (P5') and (P6')}

(P5') and (P6') follow from duality. On $x+1<\cdots<0<\cdots<x-2<x-1$ define the order reversing
map $*:i \mapsto 2x-i$. We extend this map to words $a=a_1\cdots a_h \mapsto  a_h^* \cdots a_1^*$
and affine factorizations $*: w^\beta=w^\ell \cdots w^1 \mapsto w^{1*} \cdots w^{\ell *}$, where
$w^{i*}$ is the element in $S_{\hat{x}}$ corresponding to $\content(w^i)^*$.
Note that the bracketing for the letters in the factors $w^{r+1}$ and $w^r$ used for the crystal operators 
is equivalent to bracketing all pairs $i\in w^{r+1}$ and $i+1\in w^r$, removing these mentally, then
bracketing all pairs $i\in w^{r+1}$ and $i+2\in w^r$ from the remaining letters etc.. 
This shows that $* \circ \et_r = \ft_{\ell-r} \circ *$ and  $* \circ \ft_r = \et_{\ell-r} \circ *$. 
Hence (P5') and (P6') follow from (P5) and (P6).
\end{proof}


{\large (Jennifer Morse) Department of Mathematics,
Drexel University, Philadelphia, PA 19104

{\it Email address}:\;\texttt{morsej@math.drexel.edu} }

\smallskip

{\large (Anne Schilling) Department of Mathematics, UC Davis, 
One Shields Ave., Davis, CA 95616-8633

{\it Email address}:\;\texttt{anne@math.ucdavis.edu}}

\end{document}